\documentclass[final]{article}
\usepackage[a4paper, margin = 1.15in]{geometry}

\usepackage{amsmath, amssymb, amsthm}
\usepackage{mathtools, bbm, dsfont}
\usepackage{enumitem}
\usepackage{microtype}
\usepackage[indent]{parskip}
\usepackage{graphicx}
\usepackage[english,algoruled,lined,noresetcount,norelsize]{algorithm2e}

\usepackage[colorlinks, bookmarks, linkcolor=black, citecolor=black, urlcolor=black]{hyperref}
\usepackage[hyperpageref]{backref}
\renewcommand*\backref[1]{\ifx#1\relax \else 
(Return to page #1) \fi} 
\usepackage[numbers,sort&compress]{natbib}
\usepackage[notcite,notref]{showkeys}
\def\mcite[#1]#2{\mbox{\cite[#1]{#2}}} 

\newtheorem{theorem}{Theorem}[section]
\newtheorem{lemma}[theorem]{Lemma}
\newtheorem{conj}[theorem]{Conjecture}

\newtheorem{cor}{Corollary}[theorem]
\theoremstyle{definition}
\newtheorem{claim}{Claim}[theorem]
\newtheorem{definition}[theorem]{Definition}

\newtheorem{obs}[theorem]{Observation}
\newtheorem{remark}[theorem]{Remark}

\setlist[enumerate]{itemindent=0.5\parindent+\labelwidth, 
leftmargin =*, labelindent = 0.5 \parindent}
\newcommand{\itmarab}[1]{\mbox{\rm 
({\it #1}\,\arabic{*}\hspace{0.05em})}}

\newcommand{\eps}{\varepsilon}
\newcommand{\main}{\mathrm{main}}
\newcommand{\buf}{\mathrm{buf}}
\newcommand{\Emb}{\mathrm{Emb}}
\newcommand{\im}{\mathrm{Im}}
\newcommand{\dist}{\mathrm{dist}}
\newcommand{\exth}{\delta_{\rm e}(\Delta)}
\newcommand{\bigO}{\mathcal{O}}
\newcommand{\fail}{\mathrm{fail}}
\DeclareMathOperator{\supp}{Supp}
\let\tilde\widetilde

\newcommand{\oldqed}{}
\def\endofClaim{\hfill\scalebox{.6}{$\Box$}}
\newenvironment{claimproof}[1][Proof]{
  \renewcommand{\oldqed}{\qedsymbol}
  \renewcommand{\qedsymbol}{\endofClaim}
  \begin{proof}[#1]}
  {\end{proof}
  \renewcommand{\qedsymbol}{\oldqed}} 

\title{Robustness of the Sauer--Spencer Theorem}

\author{%
Peter Allen\thanks{Department of Mathematics, The London School of
Economics, Houghton Street, London WC2A 2AE, UK. E-mail: {\tt  p.d.allen@lse.ac.uk}.}
\and
Julia B\"{o}ttcher\thanks{Department of Mathematics, The London School of
Economics, Houghton Street, London WC2A 2AE, UK. E-mail: {\tt  j.boettcher@lse.ac.uk}.}
\and
Yoshiharu Kohayakawa\thanks{Instituto de Matem\'atica e
  Estat\'{\i}stica, Universidade de S\~ao Paulo, Rua do Mat\~ao 1010,
  05508–090 S\~ao Paulo, Brazil.  E-mail: {\tt yoshi@ime.usp.br}.
  Partially supported by FAPESP (2023/03167-5) and CNPq (407970/2023-1, 315258/2023-3).}
\and
Mihir Neve\thanks{Department of Mathematics, The London School of
Economics, Houghton Street, London WC2A 2AE, UK. E-mail: {\tt  m.s.neve@lse.ac.uk}.}
}

\date{}

\begin{document}

\maketitle

\begin{abstract}
We prove a robust version of a graph embedding theorem
of Sauer and Spencer. To state this sparser analogue, we define $G(p)$
to be a random subgraph of $G$ obtained by retaining each edge of $G$
independently with probability $p \in [0,1]$, and let $m_1(H)$ be the
maximum $1$-density of a graph $H$. We show that for any constant
$\Delta$ and $\gamma > 0$, if $G$ is an $n$-vertex host graph with minimum degree
$\delta(G) \geq (1 - 1/2\Delta + \gamma)n$ and $H$ is an $n$-vertex graph with
maximum degree $\Delta(H) \leq \Delta$, then for $p \geq
Cn^{-1/m_1(H)}\log n$, the random subgraph $G(p)$ contains a copy of
$H$ with high probability. Our value for $p$ is optimal up to a
$\log$-factor. 
\vspace{5pt}

In fact, we prove this result for a more general minimum degree 
condition on~$G$, by introducing an \emph{extension
threshold}~$\exth$, such that the above result holds for graphs~$G$
with ${\delta(G) \geq (\exth + \gamma)n}$. We show that
$\exth \leq (2\Delta-1)/2\Delta$, and further conjecture that
$\exth$ equals $\Delta/(\Delta+1)$, which matches the minimum degree condition on~$G$ in the
Bollob\'as--Eldridge--Catlin Conjecture. A main tool in our proof is a vertex-spread version of the blow-up lemma of
Allen, B\"{o}ttcher, H\`{a}n, Kohayakawa, and Person, which we believe
to be of independent interest.
\end{abstract}

\vspace{10pt}
\section{Introduction}

Many of the most fundamental problems in graph theory can be phrased as embedding problems.
A graph~$H$ is said to have an \emph{embedding} into a host
graph~$G$ if there exists an injective map from $V(H)$ to $V(G)$ which
maps every edge of $H$ to some edge in $G$. For such problems, one is
usually interested in finding sufficiency conditions on the host
graph, such as bounds on its minimum degree, that allow for embedding a given family of graphs.

A prominent result of this flavour is Dirac's
theorem~\cite{Dirac_theorem}, which states that any graph on $n\geq3$
vertices with minimum degree at least~$n/2$ contains a Hamilton
cycle. Moving to other spanning subgraphs,
Bollob\'{a}s in 1978~\cite{bollobás1978extremal} conjectured that
if~$G$ is an $n$-vertex graph with minimum degree at least
$(1/2 + \eps)n$, then~$G$ contains any spanning tree of bounded
constant degree~$\Delta$ for any $\eps>0$
and~$n\geq n_0(\eps)$. This was shown to be true by Koml\'{o}s,
S\'{a}rk\"{o}zy, and Szemer\'{e}di in 1995
\cite{Tree_embedding_komlos}. In fact, they later improved their
result to show that it is possible to find any spanning tree of
degree $\mathcal{O}(n/\!\log n)$ under a similar minimum degree condition \cite{spanning_trees_n_by_logn}.

Such embedding results have been studied for a variety of other families of
subgraphs, such as clique-factors~\cite{Hajnal_szemeredi},
$F$-factors~\cite{F_factors}, powers of Hamilton
cycles~\cite{powers_of_hcycles}, and for subgraphs with bounded
degree and sublinear bandwidth~\cite{dense_bandwidth_theorem}. A
common link in all these embedding theorems is that the subgraphs have sublinear bandwidth. However, the question of which minimum degree
enforces the appearance of all spanning subgraphs of a given maximum
degree still remains open. Observe that this class of graphs also
contains expanders, which do not have sublinear bandwidth. Bollob\'as,
Eldridge, and Catlin~\cite{Bollobas_Eldridge_version_conj,
  catlin_version_conj} made the following conjecture. \footnote{ This conjecture still remains open: Kun announced a proof of the conjecture for sufficiently large $n$ in 2009, but so far no manuscript has appeared.}

\begin{conj}[Bollob\'as--Eldridge--Catlin Conjecture]
  \label{conj:BEC} Let $\Delta$ be a positive
  integer. Suppose $G$ is an $n$-vertex graph with minimum degree
  $\delta(G) \geq \tfrac{\Delta}{\Delta+1}n$ and $H$ is an $n$-vertex
  graph with maximum degree $\Delta(H) \leq \Delta$.  Then~$H$ is a
  spanning subgraph of~$G$.
\end{conj}

This conjecture is known to be tight, as is shown by a slightly unbalanced
complete $(\Delta+1)$-partite graph, which does not contain a
$K_{\Delta+1}$-factor. There have been several partial results towards
the resolution of this conjecture. For instance, this conjecture is
known to be true for $\Delta \leq 2$ \cite{BEC_Delta2,BEC_Delta22}; for
the case where $\Delta = 3$ and $n$ is large~\cite{BEC_Delta3}; and
when~$H$ is a bipartite graph~\cite{BEC_bipartite}.  In a major improvement, Kaul, Kostochka, and Yu \cite{Kaul_kostocha_yu} showed that if the graph~$H$ had a large maximum degree, then it can be embedded into a graph~$G$ having minimum degree $\delta(G) \geq (1 - 3/5(\Delta+1))n$. 

To this date,
however, the best known bound on the minimum degree that enforces all
spanning subgraphs with maximum degree $\Delta$ is given by the
following old and well-known theorem of Sauer and
Spencer~\cite{Sauer_spencer}.

\begin{theorem}[Sauer--Spencer Theorem]
  \label{thm:Sauer_Spencer} Let $\Delta > 0$ be given. Suppose that $G$ is an $n$-vertex graph with minimum degree
  $\delta(G) \geq \frac{2\Delta-1}{2\Delta}n$ and $H$ is an
  $n$-vertex graph with maximum degree~$\Delta(H) \leq \Delta$. Then
  $H$ is a spanning subgraph of $G$.
\end{theorem}

In general, embedding theorems such as the ones mentioned so far allow
for embedding subgraphs into dense graphs, that is, graph
having~$\Theta(n^2)$ edges. A popular line of research in the past few
decades has been to establish sparse analogues of such results. Roughly, the idea is to use the high minimum degree host
graph from a Dirac-type embedding theorem as a template and ask whether a sparse
spanning subgraph of this host graph will inherit its
subgraph embedding properties.
This is captured by the notion of robustness, which was proposed
by Krivelevich, Lee, and Sudakov~\cite{robust_hamiltonicity} in their
study of Hamiltonicity. 

A robustness result examines
how strongly a graph satisfies a property. To make this more precise, given a
graph~$G$ and $p\in[0,1]$, let $G(p)$ be the random subgraph
of $G$ obtained by retaining edges of $G$ independently with
probability~$p$. Further, we say that a graph property holds \emph{with high probability} if the probability of this event tends to $1$ as the size of the graph, $|G| = n$, grows to infinity.  
Suppose the graph~$G$ satisfies some
property~$\Pi$. Then this property~$\Pi$ is said to hold
\emph{robustly} for~$G$ and for some $p < 1$, if the random subgraph $G(p)$ satisfies the
property~$\Pi$ with high probability. For instance,
Krivelevich, Lee, and Sudakov~\cite{robust_hamiltonicity} proved that
Hamiltonicity is robust for Dirac graphs. They showed that if $G$ is
an $n$-vertex graph with minimum degree $n/2$ and
$p\ge Cn^{-1}\log n$, then with high probability, the random subgraph~$G(p)$ is 
Hamiltonian as well.

One motivation for establishing robustness results is that they immediately imply counting results on the number of embeddings in the dense host graph.
In the case of robust embedding theorems with
tight minimum degree conditions, these counting results reveal a discontinuous behaviour of the number of copies of the
subgraph in question. For indeed, below this minimum degree
threshold (say $n/2$ for Hamiltonicity), the host graph might not have
any copy of the fixed subgraph, but as soon as the minimum degree
condition is met, there is a sudden emergence of many copies of
this subgraph in the host graph.

Robustness results have been previously studied for some graph and hypergraph
properties. For instance, robust versions are known for the
Hajnal--Szemer\'edi theorem for containing clique
factors~\cite{robust_corradi_hajnal, toolkit_robust_thresholds}; for
the Dirac threshold for containing bounded degree
trees~\cite{bastide2024randomembeddingsboundeddegree, toolkit_robust_thresholds}; for containment of powers of Hamilton cycles~\cite{joos2024robusthamiltonicity}; for embedding subgraphs with bounded maximum degree and sublinear bandwidth~\cite{sparse_blowup}; and for containment of
Hamilton $\ell$-cycles~\cite{Alp_robustness_spread} and perfect
matchings~\cite{robust_pmatch_hyperG, toolkit_robust_thresholds} in
hypergraphs. While only a few results on robustness, as defined
above, are known, we refer the interested reader to a survey by
Sudakov~\cite{Robust_survey} on related questions of a similar
flavour.

\subsection{Our Results}

Our main result is the following robust version of the
Sauer--Spencer theorem (Theo\-rem~\ref{thm:Sauer_Spencer}). The value
of $p$ that we shall be working with is given by the \emph{maximum
  $1$-density}~$m_1(H)$ of the graph~$H$ we are embedding, which is
defined as follows. Let $d_1(H)\coloneqq e(H)/(v(H) -1)$ denote the
\emph{$1$-density} of $H$, and let
$m_1(H)\coloneqq \max_{H'\subseteq H}d_1(H')$, where the maximum runs over all
subgraphs of $H$ with at least two vertices.

\begin{theorem}[Robust Sauer--Spencer Theorem]
    \label{thm:Robust_main_basic}
    For all\/ $\gamma > 0$ and\/ $\Delta \in \mathbb{N}$, there is a
    constant\/ $C > 0$ such that if\/ $H$ is an\/ $n$-vertex graph
    with maximum degree\/ $\Delta(H) \leq \Delta$ and\/ $G$ is an\/
    $n$-vertex graph with minimum degree\/
    $\delta(G) \geq \bigl( \tfrac{2\Delta-1}{2\Delta}+ \gamma\bigr) n$,
    then for\/ $p \geq C n^{-1/m_1(H)}\log n$, with high probability, the graph\/ $H$ is
    a subgraph of~$G(p)$.
\end{theorem}

We remark that, technically, this theorem concerns a sequence $H=(H_n)_{n\in\mathbb{N}}$ of graphs, but to simplify notation
we just write $H$, as is standard in the area.

Our value for $p$ in Theorem~\ref{thm:Robust_main_basic} is optimal up
to the $\log$ factor (see Section~\ref{sec:conclusion} for details). For the proof of
Theorem~\ref{thm:Robust_main_basic} we prove and use a stronger
version of Theorem~\ref{thm:Sauer_Spencer}, where we show that, under similar conditions, it is
possible to obtain an embedding of $H$ into $G$ by \emph{extending} any
given partial embedding of a few vertices of $H$ into $G$ (see Lemma~\ref{lem:extender_sauer_spencer}). In fact, as
this ability to extend is the only r\^{o}le of Theorem~\ref{thm:Sauer_Spencer} in our proof,
we are able to prove a more general version of
Theorem~\ref{thm:Robust_main_basic}. The
following definition captures the above extension property.

\begin{definition}[Extension Threshold]
\label{defn:extension_threshold}
For all $\Delta \in \mathbb{N}$, let the \emph{extension threshold}
$\exth$ be the smallest real number~$\delta_{\rm e}$ such that the
following holds.
For every $\gamma>0$ there exists $\eta>0$ such that for all
sufficiently large~$n$, we have that if~$G$ and~$H$ are $n$-vertex graphs with
$\delta(G)\ge(\delta_{\rm e}+\gamma)n$ and with $\Delta(H)\le\Delta$, and if
$S\subset V(H)$ is a subset of size $|S|\le\eta n$ given with an embedding $\varphi_S: S\to V(G)$ of $H[S]$ into $G$, then there
is an embedding $\varphi: V(H)\to V(G)$ of~$H$ into~$G$ that
extends~$\varphi_S$. 
\end{definition}

The tightness of Conjecture~\ref{conj:BEC} implies that
$\exth\ge \Delta/(\Delta+1)$.  On the other hand, our stronger
version of the Sauer--Spencer theorem (Lemma~\ref{lem:extender_sauer_spencer}) establishes the
upper bound~$\exth\le (2\Delta-1)/2\Delta$. We prove the
following result, which immediately implies
Theorem~\ref{thm:Robust_main_basic}.

\begin{theorem}[Main Technical Theorem]
    \label{thm:Robust_main}
    For all\/ $\gamma > 0$ and\/ $\Delta \in \mathbb{N}$, there exists
    a constant\/ $C > 0$ such that if\/ $H$ is an\/ $n$-vertex graph
    with maximum degree\/ $\Delta(H) \leq \Delta$ and\/ $G$ is an\/
    $n$-vertex graph with minimum degree\/
    $\delta(G) \geq (\exth+ \gamma) n$, then for\/
    $p \geq C n^{-1/m_1(H)}\log n$, with high probability, the graph\/~$H$ is contained in the random subgraph~$G(p)$ of $G$.
\end{theorem}

We believe that if Conjecture~\ref{conj:BEC} holds true, then the value of $\exth$ should match the minimum degree condition in Conjecture~\ref{conj:BEC} and equal $\Delta/(\Delta+1)$. 
On the other hand, as in Theorem~\ref{thm:Robust_main_basic}, our value for~$p$, and hence
the strength of our robustness result, is optimal up to the $\log$
factor.

Further, it is not hard to check that among graphs~$H$ with maximum
degree $\Delta\ge2$, the quantity~$m_1(H)$ is maximised by
$K_{\Delta+1}$ (albeit not uniquely), which has
$m_1(H)=(\Delta+1)/2$. In fact, for all $H$ with $\Delta(H) \leq \Delta$ and
$m_1(H)=(\Delta+1)/2$ we show in Section~\ref{sec:conclusion} (see Theorem~\ref{thm:tight_improvement}) that it is possible to work with a slightly better
probability $p \geq Cn^{-2/(\Delta+1)}(\log n)^{1/\binom{\Delta+1}{2}}$, which
is optimal up to the
constant~\cite{FKNP_Fractional_Thresholds}. 

\subsubsection{Overview of the Proof}

In order to obtain a near-optimal value of $p$ in
Theorem~\ref{thm:Robust_main}, we use the notions of spreadness and
spread measures as introduced by
Talagrand~\cite{talagrand_spread_measure}. In the proof of the fractional version of the Kahn--Kalai conjecture, Frankston,
Kahn, Narayanan, and Park \cite{FKNP_Fractional_Thresholds} established
a relation between the existence of such a spread measure and the threshold
function for a graph property. In the case of our robustness result, this threshold function is the best value of~$p$ that works for Theorem~\ref{thm:Robust_main}.

Thus, as a first step in our proof, we show the existence of a so-called vertex-spread measure, which implies the required spread measure that is sufficient for proving Theorem~\ref{thm:Robust_main}
for the said value of $p$. 
Spread measures have been use before for proving robustness
theorems~\cite{nenadov_pham_spreadblowup, Alp_robustness_spread}. In
fact, in 2023, Pham, Sah, Sawhney, and Simkin
\cite{toolkit_robust_thresholds} proved a series of robustness results
highlighting the use of spreadness in such proofs and various methods
for constructing spread measures. However, in our setup new ideas are needed.

One of our main tools, which we use to construct this vertex-spread measure,
is a spread version of the blow-up lemma of Allen, B\"ottcher, H\`an,
Kohayakawa, and Person~\cite{sparse_blowup}. We believe this result
will be useful to approach other problems as well.
%
% The blow-up lemma from \cite{sparse_blowup} is an outcome of a more
% general version of the blow-up lemma for sparse graphs (such as
% $G(n,p)$ or bijumbled pseudorandom graphs), proving which was the
% focus of \cite{sparse_blowup}.
%
Our proof of this spread blow-up lemma results from a careful
probabilistic analysis of the key steps used in the proof of the
sparse blow-up lemma in \cite{sparse_blowup}. It should be mentioned that a spread version of the blow-up lemma of
Koml\'os, S\'ark\"ozy, and Szemer\'edi~\cite{Blowup_lemma} was proved
recently by Nenadov and Pham~\cite{nenadov_pham_spreadblowup}, but
their result is not strong enough for our proof of
Theorem~\ref{thm:Robust_main}. The main difference is that their spread blow-up lemma applies only to a bounded-size subgraph of the reduced graph of a regular partition for $G$, while our version applies to the entire reduced graph. 

The final step of our proof involves processing and
partitioning the given graphs $G$ and $H$, so as to satisfy the
sufficient conditions for applying our spread blow-up lemma. We use the extension property, as specified in
Definition~\ref{defn:extension_threshold}, to be able to partition the
graph~$H$ in agreement with the requirements of our spread blow-up
lemma.

\noindent \textbf{Organisation.} The remainder of the paper is organised as follows. We fix notation and collect some preliminary tools in Section~\ref{sec:prelims}. In Section~\ref{sec:Sauer_Spencer_extension} we prove our extension of the Sauer--Spencer theorem. This is followed by an introduction to regularity and the blow-up lemma in Section~\ref{sec:Regularity_dense_SBL}, and to spread and vertex-spread measures in Section~\ref{sec:robust}. In Section~\ref{sec:robust} we also state our spread version of the blow-up lemma, the proof of which is deferred later to Section~\ref{sec:spreadblow}. Theorem~\ref{thm:Robust_main} is proven in Section~\ref{sec:robust_overview} where, modulo two lemmas for the graphs~$G$ and~$H$, we show how Theorem~\ref{thm:Robust_main} follows from the spread blow-up lemma. Then, in Section~\ref{sec:lem_GH}, we show how the given graphs~$G$ and~$H$ can be processed and partitioned to satisfy the requirements of the spread blow-up lemma. Finally, Section~\ref{sec:conclusion} contains concluding remarks.

\section{Preliminaries}
\label{sec:prelims}
%Hajnal Szemeredi theorem extended form of Sauer Spencer
All graphs shall be assumed to be simple and finite. Most of the graph-theoretic notation that we use is standard. We shall use $|G|$ to denote the size of the vertex set of $G$. Given two disjoint subset of vertices $A, B \subseteq V(G)$, we use $E_G(A,B) \subseteq E(G)$ to denote the set of edges of $G$ with one end point in $A$ and the other in $B$. For any subset $S \subseteq V(G)$, the subgraph of $G$ induced by the vertex set $S$ shall be denoted by $G[S]$. For any subgraph $H$, an $H$-factor in $G$ is a collection of vertex-disjoint copies of $H$ in $G$, such that each vertex of $V(G)$ lies in exactly one of these copies of $H$. 

In this paper, the graph $G$ will be primarily reserved to denote the host graph, and will be paired with some fixed bounded degree spanning subgraph $H$ that needs to be embedded into $G$. The variable~$n$ will typically denote the size of the vertex sets of $H$ and $G$. For subsets $S, T \subseteq V(G)$, we use $N_G(S; T)$ to denote the neighbours, not necessarily common, of the vertices of $S$ in the graph $G$, that lie in the set $T\setminus S$. For simplicity, we use $N_G(v; T)$ when $S$ is a singleton set $\{v\}$, and $N_G(S)$ when $T = V(G)$. Further, we shall use $N^*_G(S; T)$, and analogously $N^*_G(S)$, to denote common neighbours of the vertices of $S$. The \emph{closed neighbourhood} of $S$, defined by $N(S) \cup S$, shall be denoted by $N[S]$. The subscript $G$ will be dropped in the neighbourhood notations whenever the ambient graph is clear from the context.  

Suppose $H$ and $G$ are two graphs given with a vertex map~$\varphi: V(H) \to V(G)$. Then, for any edge $e = uv$ of $H$, the notation~$\varphi(e)$ shall be used to denote the potential edge~$\varphi(u)\varphi(v)$ on the vertex set~$V(G)$. We shall say that an injective map~$\varphi: V(H) \rightarrow V(G)$ is an \emph{embedding} of $H$ into $G$, if $\varphi$ embeds a copy of $H$ into $G$ (i.e. $\varphi(E(H)) \subseteq E(G)$). Further, we denote by $\Emb(H,G)$, the set of all embeddings of $H$ into $G$. Finally, given a graph $G$, the \emph{blow-up} of $G$ with respect to some function~${f: V(G) \to \mathbb{N}}$ is defined to be a graph~$G^*$ obtained by replacing each vertex $v \in V(G)$ with an independent set of size $f(v)$ and each edge~$uv \in E(G)$ with a complete bipartite graph~$K_{f(u), f(v)}$ having parts of size $f(u)$ and $f(v)$, respectively.

We will require the following theorem of Hajnal and Szemer\'{e}di~\cite{Hajnal_szemeredi}, which allows for partitioning a graph with bounded maximum degree into independent sets of roughly equal sizes. Given an $n$-vertex graph~$H$ and a positive integer~$k$, the graph~$H$ is said to have an \emph{equitable partition} into $k$ sets, if $V(H)$ affords a partition~$V_1 \sqcup \dots \sqcup V_k$ such that $|V_i| = \lfloor n/k \rfloor$ or $\lceil n/k\rceil$ for all $i \in [k]$.

\begin{theorem}[Hajnal--Szemer\'{e}di Theorem]
\label{thm:Hajnal-Szemeredi}
Let $\Delta \in \mathbb{N}$ be given. If\/ $H$ is a graph with maximum degree~$\Delta(H) \leq \Delta$, then\/ $H$ has an equitable partition into\/ $(\Delta + 1)$ independent sets.     
\end{theorem}

Theorem~\ref{thm:Hajnal-Szemeredi} has many useful applications. For instance, given a graph~$H$ with $\Delta(H) \leq \Delta$ and some constant~$C$, we shall use Theorem~\ref{thm:Hajnal-Szemeredi} to obtain an equitable partition of $V(H)$ with parts of size~$\Theta(n)$ and such that for every pair of vertices~$x,y$ within a part, we have $\dist_H(x,y) \geq C$. Further, when applied on the complement graph, Theorem~\ref{thm:Hajnal-Szemeredi} returns a collection of vertex-disjoint cliques $K_r$ that cover almost all vertices of the given graph. This is stated more precisely in the following corollary of the Hajnal--Szemer\'{e}di theorem.

\begin{cor}
    \label{cor:clique_hajnal_szemeredi}
    Let the constant $r \in \mathbb{N}$ be given. If $G$ is an $n$-vertex graph with minimum degree $\delta(G) \geq (1 - 1/r)\,n$, then $G$ contains a $K_r$-factor covering all but at most $r-1$, vertices of $G$. 
\end{cor}

\section{An Extension of the Sauer--Spencer Theorem}
\label{sec:Sauer_Spencer_extension}

In this section, we prove the claimed upper bound on the extension threshold $\exth$.

The graph packing theorem of Sauer and Spencer (Theorem~\ref{thm:Sauer_Spencer}), while requiring a higher minimum degree condition on $G$ than Corollary~\ref{cor:clique_hajnal_szemeredi}, allows for embedding any bounded degree spanning subgraph~$H$ into the host graph~$G$. Its proof, as given in \cite{Sauer_spencer}, relies on a simple but clever vertex switching argument, which algorithmically finds an embedding map of $H$ into $G$. 

More precisely, the algorithm begins with an arbitrarily chosen bijection from~$V(H)$ to~$V(G)$. If this map is not yet an embedding of~$H$ into~$G$, then there is an edge~$e \in E(H)$ whose image is not in~$E(G)$. The degree conditions on~$G$ and~$H$ then facilitate swapping the image of one of the end points of~$e$ with the image of some other vertex of $H$, such that the edge~$e$ gets successfully embedded into an edge of~$G$ without affecting any of the so-far correctly embedded edges of~$H$. Repeating this process a finite number of times returns a desired embedding of~$H$ into~$G$.

We prove a stronger version of Theorem~\ref{thm:Sauer_Spencer}. For this, let $S \subseteq V(H)$ be a small fraction of the vertices of $H$. We show that in the setting of the Sauer--Spencer theorem, if we are additionally given an embedding~$\varphi_S$ of~$H[S]$ into~$G$, that is a partial embedding of the graph~$H$ into~$G$, then the switching argument outlined above can be used to extend~$\varphi_S$ to an embedding of the entire graph~$H$ into~$G$. More precisely, we obtain the following lemma.

\begin{lemma}
    \label{lem:extender_sauer_spencer}
    For all\/ $\Delta \in \mathbb{N}$ and\/ $\gamma \in (0, 1/2\Delta)$, 
    let\/ $G$ be an $n$-vertex graph with\/ $\delta(G) \geq \bigl(\frac{2\Delta-1}{2\Delta} + \gamma \bigr)n$ and let $H$ be an $n$-vertex graph with\/ $\Delta(H) \leq \Delta$. 
    Suppose\/ $S \subseteq V(H)$ is a set of vertices of size\/ $|S| \leq \gamma \Delta n$ and\/ 
    $\varphi_S: S \to V(G)$ is an embedding of\/ $H[S]$ into\/ $G$. 
    Then, there exists an embedding\/ $\varphi \in \Emb(H,G)$ of\/ $H$ into\/ $G$ that extends the given embedding $\varphi_S$. 
\end{lemma}

In particular, this lemma establishes that the extension threshold satisfies $\exth\leq (2\Delta-1)/2\Delta$.

\begin{proof}
    Given $\Delta \in \mathbb{N}$ and $\gamma \in (0, 1/2\Delta)$, let $S \subseteq V(H)$ and the map~$\varphi_S : S \to V(G)$ be given in accordance with the conditions of the lemma. 
    For ease of notation, set $\Lambda = (2\Delta - 1)/2\Delta$, and let $H'$ and $G'$ denote the graphs $H - S$ and $G - \im(\varphi_S)$, respectively. 
    We consider the set of bijections $$\Phi_S = \bigl\{\varphi : V(H) \to V(G) : \varphi|^{\phantom0}_S = \varphi_S\bigr\},$$ 
    which extend the given map $\varphi_S: S \to \im(\varphi_S)$. Note that the elements of $\Phi_S$ need not embed $H$ into $G$. 
    For any bijection~$\varphi: V(H) \to V(G)$, we shall call an edge~$xy \in E(H)$ to be \emph{unmapped} by $\varphi$ if $\varphi(x)\varphi(y) \notin E(G)$, and \emph{mapped} otherwise. 
    Let~$\varphi \in \Phi_S$ be a bijection that extends $\varphi_S$ and maximises the number of \emph{mapped} edges in $H$. 
    We show that $\varphi$ is an embedding of $H$ into $G$, as required.

    Suppose for a contradiction that $xx^* \in E(H)$ is an edge \emph{unmapped} by $\varphi$.
    As $\varphi$ extends $\varphi_S$, and $\varphi_S$ embeds $H[S]$ into $G$, every edge of $E\bigl(H[S]\bigr)$ is \emph{mapped} by $\varphi$. 
    Thus, we can assume without loss of generality that $x \in V(H')$. 
    Consider the set of vertices $T = \varphi(N_H(x)) \subseteq V(G)$. Note that $|T| \leq \Delta$ as the graph $H$ has maximum degree $\Delta(H) \leq \Delta$.
    Then, as each vertex in $G$ has at most $(1 - \Lambda - \gamma)n$ non-neighbours in $G'$, and as $|V(G')| = n - |S| \geq (1 - \gamma \Delta)n$, we have that
    the number of common neighbours of $T$ in $V(G')$ is at least 
    \begin{eqnarray*}
        \bigl|N^*_G(T; V(G'))\bigr| &\geq& \bigl((1-\gamma\Delta) - \Delta(1 - \Lambda - \gamma)\bigr)n\\
        &=&\bigl(1 - \Delta(1 - \Lambda)\bigr) = \bigl(1 - \Delta\cdot (1/2\Delta)\bigr)n = n/2.
    \end{eqnarray*}
    For any $y \in V(H')$ with $\varphi(y) \in N^*_G(T; V(G'))$, define the bijection $\varphi_y^*: V(H) \to V(G)$ as follows. 
    $$\varphi^*_y(z) = 
    \begin{cases}
        \varphi(y) & \text{if }z = x,\\
        \varphi(x) & \text{if }z = y,\\
        \varphi(z) & \text{otherwise.}
    \end{cases}$$
    \indent Note that $\varphi_y^*$ is obtained from $\varphi$ by swapping the images of $x$ and $y$, where $x,y \in V(H)$. 
    As $\varphi(x), \varphi(y) \notin \im(\varphi_S)$, it follows that the map $\varphi_y^*$ is an element of $\Phi_S$. 
    Further, as $\varphi(y) \in N^*_G(T; V(G'))$, we have that for every neighbour $a \in N_H(x)$ the edge $ax \in E(H)$ is \emph{mapped} by $\varphi_y^*$. More  specifically, the edge $xx^* \in E(H)$ gets \emph{mapped} by $\varphi_y^*$.
    Thus, if it were the case that every edge $yz \in E(H)$ \emph{mapped} by $\varphi$ is also \emph{mapped} by $\varphi_y^*$, 
    then the bijection $\varphi_y^* \in \Phi_S$ would have strictly more \emph{mapped} edges than $\varphi$, thereby contradicting the maximality assumption on $\varphi$. 

    Hence, for every $y \in V(H')$ with $\varphi(y) \in N^*_G(T; V(G'))$ there exists a vertex~$z \in V(H)$ such that the edge~$yz \in E(H)$ is \emph{mapped} by $\varphi$ but not by $\varphi_y^*$.
    Stated otherwise, for each $y \in V(H')$ with $\varphi(y) \in N^*_G(T; V(G'))$, there exists $z \in V(H)$ such that $yz \in E(H)$, $\varphi(y)\varphi(z) \in E(G)$, but $\varphi(x)\varphi(z) \notin E(G)$. 
    However, $\varphi(x)$ has at most $(1 - \Lambda - \gamma)n$ non-neighbours in $G$ which can serve as such a $\varphi(z)$, 
    and each corresponding $z$ has at most $\Delta$ neighbours in $H$ which could serve as $y$. Thus we get the following upper bound on the size of the common neighbourhood of $T$ in $G'$.
    $$\bigl|N^*_G(T; V(G'))\bigr| \leq \Delta \bigl(1 - \Lambda - \gamma\bigr)n = \bigl(1/2 - \gamma \Delta\bigr)n < n/2.$$
    This contradicts the lower bound on $\bigl|N^*_G(T; V(G'))\bigr|$ obtained above and hence, by contradiction, all edges of $E(H)$ must be \emph{mapped} by $\varphi$. 
    Thus, $\varphi$ is the required embedding which extends the given partial embedding $\varphi_S$.
\end{proof}

\section{Regularity and the Blow-up Lemma}
\label{sec:Regularity_dense_SBL}

We now shift our focus to the notion of regularity. Given a graph~$G$ and non-empty disjoint subsets~$A, B \subseteq V(G)$, the \emph{density} of $(A,B)$ in $G$ is defined to be $d(A,B)\coloneqq |E_G(A,B)|/{|A||B|}$. The pair $(A,B)$ is said to be \emph{$\eps$-regular} in $G$ if for every $A'\subseteq A$ and $B'\subseteq B$ with $|A'| \geq \eps|A|$ and $|B'| \geq \eps|B|$, we have that $|d(A',B') - d(A,B)| \leq \eps$. 

Further, an $\eps$-regular pair $(A,B)$ is said to be \emph{$(\eps, d)$-regular} if additionally we have $d(A,B) \geq d-\eps$. We should point out that the latter requirement for $d(A,B)$, while not common, has been chosen to match the blow-up lemma in \cite{sparse_blowup}, where the authors used it to make the so-called regularity inheritance properties much easier to deal with. Some of the theorems stated below (Theorem~\ref{thm:regularity} to Lemma~\ref{lem:robust_regularity}) have been originally proven for the version of the definition requiring $d(A,B) \geq d$ instead of $d-\eps$. However, in all these cases, it is not difficult to see that these statements hold for when $(\eps,d)$-regularity requires $d(A,B) \geq d - \eps$ as well. 

In his celebrated \emph{regularity lemma}, Szemer\'{e}di~\cite{Sz_regularity_OG} showed that the vertex set of any large graph can be equitably partitioned into a bounded number of parts such that most pairs of parts are $\eps$-regular. In this paper, we will use the following equivalent degree-form of Szemer\'{e}di's regularity lemma (see, for example, \mcite[Theorem 1.10]{komlos_regularity_survey}).    

\begin{theorem}[Regularity Lemma]
  \label{thm:regularity}
  For every $\eps>0$ and every integer $m_0$ there is an $M_0=M_0(\eps,m_0)$ such that for every $d\in[0,1]$ and for every graph~$G$ on at least $M_0$ vertices there exist a partition~$V_0 \sqcup V_1\sqcup \dots \sqcup V_m$ of $V(G)$ and a spanning subgraph~$G'$ of $G$ such that the following hold:
  \begin{enumerate}[label=\itmarab{RL}]
  \item \label{itm:RL1} $m_0 \leq m \leq M_0$,
  \item \label{itm:RL2} $d_{G'}(x)>d_G(x)-(d+\eps)|V(G)|$ for all vertices $x\in V(G)$,
  \item \label{itm:RL3}for all $i \geq 1$ the induced subgraph $G'[V_i]$ is empty,
  \item \label{itm:RL4} $|V_0|\leq\eps|V(G)|$ and $|V_1|=|V_2|=\dots=|V_m|$,
  \item \label{itm:RL5} for all $1\leq i<j\leq m$, either 
$(V_i,V_j)$ is $(\eps,d)$-regular or $E_{G'}(V_i,V_j)$ is empty.
  \end{enumerate}
\end{theorem}

Among other merits, the degree form of the regularity lemma allows for easy transference of the minimum degree condition on~$G$ to an auxiliary \emph{reduced graph}~$R$ generated by the regular partition, which is useful for our problem. Given the partition $V_0 \sqcup V_1\sqcup\dots\sqcup V_m$, as provided by Theorem~\ref{thm:regularity}, a reduced graph~$R$ for this partition can be naturally constructed as follows. We let the vertex set of~$R$ be $[m]$ and let the edge set contain edges~$ij$ for $1\leq i,j\leq m$ for exactly those pairs $(V_i,V_j)$ that are $(\eps,d)$-regular in~$G'$. Thus,~$ij$ is an edge of~$R$ if and only if~$G'$ has an edge between $V_i$ and $V_j$. To avoid confusion, the vertices of the reduced graph will sometimes be referred to as \emph{nodes} instead of vertices. The following corollary of Theorem~\ref{thm:regularity} demonstrates the transference of the minimum degree condition on~$G$ to the reduced graph~$R$ (see, for example, \mcite[Proposition 9]{kuhn_osthus_planar_subgraph} for a proof).

\begin{cor}
  \label{cor:reduced-min-degree}
  For every $\gamma>0$ there exist $d_0, \eps_0>0$ such that for every $0<\eps\leq \eps_0$, every $0 < d \leq d_0$ and every integer $m_0$, there exists $M_0$ so that the
  following holds.

  For every $c \geq 0$, an application of Theorem~\ref{thm:regularity} to a graph~$G$ of minimum degree at least $(c +\gamma)|V(G)|$ yields a partition $V_0, V_1,\dots,V_m$ of $V(G)$ and a subgraph $G'$ of $G$ so that additionally to properties \ref{itm:RL1}--\ref{itm:RL5}, we have,
  \begin{enumerate}[label=\itmarab{RL}, start = 6]
  \item \label{itm:RL6} the reduced graph~$R$ has minimum degree at least
    $(c +\gamma/2)m$.
  \end{enumerate}
\end{cor}

The blow-up lemma and many applications of the regularity lemma often require a stronger version of regularity, called \emph{super-regularity}. We say that an $(\eps, d)$-regular pair $(A, B)$ is $(\eps,d)$\emph{-super-regular} if every $a \in A$ has at least $(d-\eps)|B|$ neighbours in~$B$, and every $b \in B$ has at least $(d-\eps)|A|$ neighbours in~$A$. It is not difficult to show that every $(\eps,d)$-regular pair contains a \emph{large} super-regular sub-pair. This property, when applied repeatedly to the $(\eps, d)$-regular pairs encoded by a bounded degree subgraph of a reduced graph~$R$, yields the following useful lemma. We refer to \mcite[Proposition 8]{kuhn_osthus_planar_subgraph} for a simple proof.

\begin{lemma}
  \label{lem:super-regularity}
  Given a graph~$G$, let~$R$ be the reduced graph for an $(\eps, d)$-regular partition of~$G$, and let~$S$ be a subgraph of the reduced graph~$R$ with $\Delta(S)\leq \Delta$. 
  
  Then for each vertex $i$ of~$S$, the corresponding set $V_i$ contains a subset~$V'_i$ of size $(1-\eps\Delta)|V_i|$ such that for every edge $\{i,j\}\in E(S)$ the pair $(V'_i,V'_j)$ is ${(\eps/(1-\eps\Delta), d-\eps(\Delta+1))}$-super-regular.  Moreover, for every edge $\{i,j\} \in E(R)$, the pair $(V'_i,V'_j)$ is still $(\eps/(1-\eps\Delta), d-\eps(\Delta+1))$-regular.
\end{lemma}

Finally, we require the following useful observation, which roughly states that the notion of regularity is `robust' (used with its usual semantic denotation) in view of small alterations of the respective vertex sets. We refer to \mcite[Proposition~8]{Boettcher_robust_regularity_proof} for a short proof.

\begin{lemma}
  \label{lem:robust_regularity}
  Let\/ $(A,B)$ be an\/ $(\eps,d)$-regular pair and let\/ $(\hat{A},\hat{B})$ be a pair such that\/ $|\hat{A} \Delta A|\leq \alpha|\hat{A}|$ and\/ $|\hat B\Delta B|\leq \beta|\hat B|$ for some\/ $0\leq \alpha,\beta\leq 1$.  Then,\/ $(\hat{A},\hat{B})$ is an\/ $(\smash{\hat{\eps}},\smash{\hat{d}})$-regular pair for
  \begin{equation*}
\hat{\eps}:=\eps+3\big(\sqrt{{\alpha}}+\sqrt{{\beta}}\big) \quad \text{and} \quad \hat{d}:=d-2({\alpha}+{\beta}).
  \end{equation*}
  Moreover, if\/ $(A,B)$ is\/ $(\eps,d)$-super-regular and each vertex\/ $a$ in\/ $\hat{A}$ has at least\/ $d|\hat{B}|$ neighbours in~$\hat{B}$, and each vertex\/ $b$ in\/ $\hat{B}$ has at least\/ $d|\hat{A}|$ neighbours in~$\hat{A}$, then\/ $(\hat{A},\hat{B})$ is\/ $(\smash{\hat{\eps}},\smash{\hat{d}})$-super-regular with~$\smash{\hat{\eps}}$ and~$\smash{\hat{d}}$ as above.
\end{lemma}

Many applications of the regularity lemma exploit the fact that the reduced graph~$R$ inherits some properties of the graph~$G$, and hence, these properties of the reduced graph can be discovered back in the original graph~$G$. Containment under the subgraph relation is one such property, which is the essence of the \emph{blow-up lemma}, proved by Koml\'{o}s, S\'{a}rk\"{o}zy, and Szemer\'{e}di in 1997~\cite{Blowup_lemma}. The blow-up lemma from \cite{Blowup_lemma} can be stated as follows. 

\begin{lemma}[The Blow-up Lemma of Koml\'{o}s, S\'{a}rk\"{o}zy, and Szemer\'{e}di]
     For every $r, \Delta \in \mathbb{N}$ and $d > 0$, there exists a constant~$\eps > 0$ such that the following holds. \\ 
     \indent Suppose $G$ is a graph given with an $r$-partition $V_1 \sqcup V_2 \sqcup \dots \sqcup V_r$ of $V(G)$. Let $R$ be a reduced graph for $G$ such that $ij \in E(R)$ if and only if $(V_i, V_j)$ is an $(\eps, d)$-super-regular pair in~$G$, and let~$R^*$ be a blow-up of~$R$ with respect to the function $f:i\mapsto |V_i|$. Then for any graph~$H$ with $\Delta(H) \leq \Delta$, if~$H$ is a subgraph of $R^*$, then $H$ is a subgraph of $G$.
\end{lemma}
   
In this paper, we shall work with a version of the blow-up lemma proved by Allen, B\"{o}ttcher, H\`{a}n, Kohayakawa, and Person~\cite{sparse_blowup}. Unlike the blow-up lemma of Koml\'os, S\'ark\"ozy, and Szemer\'edi, the blow-up lemma from \cite{sparse_blowup} is applicable to the entire reduced graph $R$, which is a consequence of the difference in the order of quantifiers in the statements of both these versions. However, in order to achieve this variation, the blow-up lemma from \cite{sparse_blowup} works with two separate reduced graphs for~$G$: a reduced graph~$R$ that captures the regular pairs in a regular partition of~$G$, and a bounded degree spanning subgraph~$R'$ of~$R$, such that super-regularity is only required on those regular pairs of~$G$ that are encoded by~$E(R')$. 

For the remainder of this paper, we use the term `blow-up lemma' to refer to the version of the blow-up lemma from~\cite{sparse_blowup}. In order to state it, we first require some additional definitions. The setting in which we work is as follows. Let~$G$ and~$H$ be two $n$-vertex graphs, given with partitions
$\mathcal{V}=\{V_i\}_{i\in[r]}$ and $\mathcal{X}=\{X_i\}_{i\in[r]}$ of their respective vertex sets. Let $R$ be an $r$-vertex reduced graph for the partition $\mathcal{V}$ of the vertices of $G$, and let $R'$ be a spanning subgraph of $R$.

\begin{definition}
\label{defn:sparseBL_generic_terms} 
For graphs $R$ and $R'$ on $r$-vertices, and the vertex partitions $(G,\mathcal{V})$ and $(H, \mathcal{X})$,

    \begin{itemize}
        \item $\mathcal{V}$ and~$\mathcal{X}$ are said to be \emph{size-compatible} if $|V_i|=|X_i|$ for all $i\in[r]$.
        \item For $\kappa\geq 1$, we say that $(G,\mathcal{V})$ is
        \emph{$\kappa$-balanced} if there exists $m\in\mathbb{N}$ such that $m\leq |V_i|\leq \kappa m$ for all $i, j\in[r]$ (similarly defined for the partition $(H, \mathcal{X})$).

         \item The partition $(G,\mathcal{V})$ is \emph{$(\eps,d)$-regular on $R$} if for each $ij\in E(R)$, the pair $(V_i,V_j)$ is $(\eps,d)$-regular.

         \item Further, the partition $(G,\mathcal{V})$ is said to be \emph{$(\eps,d)$-super-regular on $R'$} if for every $ij\in E(R')$, the pair $(V_i,V_j)$ is $(\eps,d)$-super-regular.

    \end{itemize}
  In this case we also say that~$R$ is a \emph{reduced graph} for the partition~$\mathcal{V}$. For the graph $H$, we shall require something slightly stronger: 
  \begin{itemize}
    \item $(H,\mathcal{X})$ is an \emph{$R$-partition} if each part $X_i \in \mathcal{X}$ is non-empty, and whenever there are edges of~$H$ between $X_i$ and $X_j$, the pair $ij$ is an edge of~$R$.
 \end{itemize}
\end{definition}

Given the partitions $(G, \mathcal{V})$ and $(H, \mathcal{X})$ as above, the blow up lemma embeds the graph $H$ into $G$ by mapping vertices of $X_i$ to distinct vertices of $V_i$ for each $i \in [r]$. Note that the size compatibility of $(G, \mathcal{V})$ and $(H, \mathcal{X})$ is a necessary condition for such a mapping.  

Next we introduce the notion of buffer vertices. The blow-up lemma from \cite{sparse_blowup} is a special case of a more general version for sparse graphs, proving which was the focus of \cite{sparse_blowup}. This sparse version of the blow-up lemma embeds a copy of the graph~$H$ into $G$ when the graphs satisfy certain conditions, as were outlined in \cite{sparse_blowup}. Their proof employs a random greedy algorithm to embed most vertices of $H$ into $G$, leaving behind a fraction of unembedded vertices which are later embedded via Hall's theorem to obtain a complete copy of $H$ in $G$. For this latter step to work, the proof sets aside a few \emph{buffer} vertices, having some \emph{good} properties which facilitate the use of Hall's theorem. In order for the proof to choose and set aside these buffer vertices, the blow-up lemma requires as input, a set of \emph{potential buffer vertices} for which the following properties hold.  

\begin{definition}[Buffer Vertices]
\label{defn:buffer}
  Suppose~$R$ and $R'\subseteq R$ are graphs on~$r$ vertices, $(H,\mathcal{X})$ is an
  $R$-partition, and $(G,\mathcal{V})$ is a size-compatible partition which is $(\eps,d)$-regular on~$R$.
  We say the family $\tilde{\mathcal{X}}=\{\tilde{X}_i\}_{i\in[r]}$ of subsets of $V(H)$ is
  an \emph{$(\alpha,R')$-buffer} for $(H,\mathcal{X})$ if
 \begin{itemize}
 \item for each $i\in[r]$, we have~$\tilde{X}_i\subseteq X_i$ and that $\tilde{X}_i$ has size $|\tilde{X}_i|\ge\alpha |X_i|$, and 
 \item for each $i\in[r]$ and each $x\in\tilde{X}_i$, the first and second
   neighbourhood of~$x$ \emph{go along~$R'$}, that is,
   for each $xy, yz\in E(H)$ with $y\in X_j$ and $z\in X_k$ we have $ij\in E(R')$ and $jk\in E(R')$.
 \end{itemize}
 The vertices in $\tilde{\mathcal{X}}$ are called \emph{potential buffer vertices}. 
\end{definition}

Finally, the blow-up lemma, similar to the blow-up lemma of Koml\'os, S\'ark\"ozy, and Szemer\'edi~\cite{Blowup_lemma}, also affords image restrictions. This means that it is possible to choose subsets $I_x \subseteq V(G)$ for a few \emph{image-restricted} vertices $x \in V(H)$, and apply the blow-up lemma to obtain an embedding of~$H$ into~$G$ that embeds each image-restricted vertex~$x$ to some vertex~$v$ in $I_x$. Note that as the blow-up lemma maps each $x\in X_i$ to some vertex in~$V_i$, the vertex~$x \in X_i$ is image-restricted only if~$I_x$ is a proper subset of~$V_i$. Similarly, for non image-restricted vertices $x \in X_i$, we naturally require that $I_x = V_i$. The blow-up lemma can sustain the following form of image restrictions.   

\begin{definition}[Image Restrictions]
    We say that $\mathcal{I} = \{I_x\}_{x \in V(H)}$ is a family of \emph{$(\rho, \zeta)$-image-restrictions} if for each $i \in [r]$, at most $\rho|X_i|$ vertices in~$X_i$ are \emph{image-restricted} (that is, have $I_x \subsetneq V_i$), and for each such image-restricted vertex $x \in X_i$, the image restriction has size $|I_x| \geq \zeta |V_i|$. 
\end{definition}

With this, the blow-up lemma can be stated as follows \mcite[Lemma 7.1]{sparse_blowup}. We shall use a spread-version of this statement (see Theorem~\ref{thm:spread_dense_blowup}) in the proof of Theorem~\ref{thm:Robust_main}.

\begin{lemma}[The Blow-up Lemma]\label{lem:dense_blowup}
  For all $\Delta\ge 2$, $\Delta_{R'} \in \mathbb{N}$, $\alpha, \zeta, d>0$, and ${\kappa>1}$, there exists $\eps, \rho > 0$ such that for all $r_1 > 0$, the following holds: 
  
  Let~$R$ be a graph on $r\le r_1$ vertices and let $R'\subseteq R$ be a spanning subgraph with $\Delta(R')\leq \Delta_{R'}$. Let~$H$ and~$G$ be $n$-vertex graphs, given with $\kappa$-balanced, size-compatible vertex partitions $\mathcal{X}=\{X_i\}_{i\in[r]}$ and $\mathcal{V}=\{V_i\}_{i\in[r]}$, respectively, which have parts of size at
  least~$n/(\kappa r_1)$. Let 
  $\tilde{\mathcal{X}}=\{\tilde{X}_i\}_{i\in[r]}$ be a family of subsets of $V(H)$, and $\mathcal{I} = \{I_x\}_{x \in V(H)}$ be a family of subsets of $V(G)$.  
  Suppose that
  \begin{enumerate}[label=\itmarab{BL}]
  \item \label{itm:BL1} the partition $(G,\mathcal{V})$ is $(\eps,d)$-regular on $R$ and $(\eps,d)$-super-regular on $R'$,
  \item \label{itm:BL2} $\Delta(H)\leq \Delta$ and $(H,\mathcal{X})$ is an $R$-partition,
  \item \label{itm:BL3} $\tilde{\mathcal{X}}$ is an $(\alpha,R')$-buffer for $(H,\mathcal{X})$, and
  \item \label{itm:BL4} $\mathcal{I} = \{I_x\}_{x \in V(H)}$ is a family of $(\rho, \zeta)$-image-restrictions. 
  \end{enumerate}
  
  \noindent Then there is an embedding $\varphi \colon V(H)\to V(G)$ of~$H$ into~$G$, such that~$\varphi(x) \in I_x$ for all~$x \in V(H)$.
\end{lemma}

\section{Spread Measures and a Spread Blow-up Lemma}
\label{sec:robust}

In this section, we introduce the notion of spread measures and state our spread version of the blow-up lemma which we shall use in the proof of Theorem~\ref{thm:Robust_main}. Given a finite set~$X$ and a family~$\mathcal{F}$ of subsets of $X$, the notion of a \emph{$q$-spread} probability measure on the family $\mathcal{F}$, as introduced by Talagrand~\cite{talagrand_spread_measure}, can be defined as follows. For any subset $S \subseteq X$ of the elements of $X$, we use $\langle S \rangle$ to denote the \emph{up-set} of~$S$, defined by $\langle S \rangle \coloneqq \{A \subseteq X : S \subseteq A\}$.

\begin{definition}[$q$-spread Probability Measure]
\label{defn:spread_measure}
    Consider a finite set\/~$X$, and let\/ $\mathcal{F}$ be a family of subsets of\/ $X$. Then for\/ $q > 0$, a probability measure\/~$\mu$ supported on\/ $\mathcal{F}$, is said to be \emph{$q$-spread} if for every subset\/ $S \subseteq X$, we have $$\mu\bigl(\langle S \rangle\bigr) = \mu\bigl(\langle S\rangle \cap \mathcal{F}\bigr) = \mu\bigl(\{A \in \mathcal{F}: S \subseteq A\}\bigr) \leq q^{|S|}.$$
\end{definition}

In other words, for constructing a~$q$-spread measure, one needs to be able to assign weights (normalised to~$1$) to every subset in the family~$\mathcal{F}$ in a manner that forbids~$\mathcal{F}$ from being concentrated over any subset~$S$ of elements of~$X$.  
Thus, a $q$-spread measure on~$\mathcal{F}$ indicates the degree to which the given family~$\mathcal{F}$ of subsets of~$X$ is distributed over the ground set~$X$. 
Indeed, while every family affords a trivial $1$-spread probability measure, the existence of a $q$-spread measure on~$\mathcal{F}$ for a smaller value of~$q$ is indicative of a better distributed family~$\mathcal{F}$.  

In \cite{FKNP_Fractional_Thresholds}, Frankston, Kahn, Narayanan, and Park, showed that the existence of a $q$-spread probability measure on~$\mathcal{F}$ provides an upper bound on the threshold for a random subset of~$X$ to contain some element of~$\mathcal{F}$. More precisely, let~$X(p)$ denote a random subset of~$X$, with each element of~$X$ being retained independently with probability~$p$, and let~$\ell(\mathcal{F})$ be the largest cardinality of the minimal elements of~$\mathcal{F}$. Then the following is true \cite[Theorem 1.6]{FKNP_Fractional_Thresholds}.

\begin{theorem}
    \label{thm:FKNP_spread_measure}
    There exists a universal constant~$K$ such that for any finite set~$X$ and a family~$\mathcal{F}$ of subsets of~$X$, if there exists a \/$q$-spread probability measure on~$\mathcal{F}$, then for\/ $p \geq K q \log{\ell(\mathcal{F})}$, the binomial random subset\/ $X(p)$ contains an element of\/ $\mathcal{F}$ with probability $1 - o_{\ell}(1)$.
\end{theorem} 

\begin{remark}
    It should be pointed out that technically, \mcite[Theorem 1.6]{FKNP_Fractional_Thresholds} states and proves the above result specifically for uniform probability measures on the set ${\mathcal{F}}$ that are $q$-spread. The proof of this theorem works with a different notion of $k$-spread measures defined for a hypergraph~$\mathcal{H}$ with vertex set~$X$, and is shown to be equivalent to having a uniform $1/k$-spread measure on the family $\mathcal{F} = E(\mathcal{H})$ of the edges in $\mathcal{H}$. However, this result can be generalised to all measures as follows, as was briefly explained in the proof of \mcite[Theorem 1.1]{FKNP_Fractional_Thresholds}. 
    \newline \indent The definition of spreadness in hypergraphs, as defined in~\cite{FKNP_Fractional_Thresholds}, allows for the hypergraph~$\mathcal{H}$ to have multi-edges (many identical copies of the same edge). Thus, for a finite set $X$, any rational-valued $q$-spread probability measure $\mu$ on $\mathcal{F}$ can be converted to a uniform measure by replacing each set $F \in \mathcal{F}$ with $m!\cdot \mu(F)$ identical copies of~$F$, for some sufficiently large $m$. This generates a multi-hypergraph with a uniform measure of the same spread, and so, \mcite[Theorem 1.6]{FKNP_Fractional_Thresholds} extends to such rational-valued measures. Finally, any measure on a finite set can be slightly perturbed to obtain a rational-valued measure of roughly the same spread, and hence, the result holds true for all measures, as stated in Theorem~\ref{thm:FKNP_spread_measure}.
\end{remark}

For the purpose of our problem, we consider the ground set $X = X_G$ to be the edge set~$E(G)$ of the host graph~$G$. Consequently, the family~$\mathcal{F} = \mathcal{F}_H$ will contain as elements, those subsets of $E(G)$ which form a copy of the graph~$H$ in~$G$. Thus, to apply Theorem~\ref{thm:FKNP_spread_measure}, we need to assign a probability measure~$\mu$ to the set of all copies of~$H$ in~$G$, such that~$\mu$ is~$q$-spread for an aptly chosen~$q$. In order to prove Theorem \ref{thm:Robust_main}, we would require $q = \mathcal{O}(n^{-1/m_1(H)})$, as $\log{\ell(\mathcal{F}_H)} = \mathcal{O}(\log n)$. 

\subsection{Vertex-spread Measures}

For our problem, it is helpful to work with the following notion of \emph{vertex-spread} probability measures, which were specifically designed for working with graph embeddings by Pham, Sah, Sawhney, and Simkin in~\cite{toolkit_robust_thresholds}. Here the probability measures under consideration shall be defined over the set~$\Emb(H,G)$ of all embeddings of~$H$ into~$G$. 

\begin{definition}[$q$-vertex-spread Probability Measure]
    \label{defn:vtx_spread_measure}
    Let~$\mu$ be a probability measure on the set~$\Emb(H,G)$ for graphs~$H$ and~$G$. For $q > 0$, we say that~$\mu$ is a $q$-\emph{vertex-spread probability measure} if for all $k \in \mathbb{N}$, and for every two sequences of distinct vertices ${x_1, x_2, \dots, x_k \in V(H)}$ and ${v_1, v_2, \dots, v_k \in V(G)}$, we have $$\mu \Bigl(\bigl\{ \varphi \in \Emb(H,G) \bigm| \varphi(x_i) = v_i \text{ for all } i \in [k] \bigr\}\Bigr)  \leq \ q^k.$$
\end{definition}

Given such a vertex-spread measure on $\Emb(H,G)$, it is possible to obtain a spread measure on~$\mathcal{F}_H$ as any unlabelled copy of~$H$ in~$G$ corresponds to $|\mathrm{Aut}(H)|$ distinct embeddings of $H$ into $G$, where~$\mathrm{Aut}(H)$ denotes the set of automorphisms of the graph~$H$. The following theorem, a special case of Proposition 1.17 from~\cite{Alp_robustness_spread}, demonstrates this link between the two notions for the given problem and shows that if $\Delta(H) \leq \Delta$, then a $\mathcal{O}(1/n)$-vertex-spread measure on $\Emb(H,G)$ returns a spread measure with the desired spread. We include a short proof for completeness. 

\begin{theorem}
    \label{thm:vtx_spread_implies_spread}
    Let $H$ be an $n$-vertex graph with~$\Delta(H) \leq \Delta$, and~$G$ be any $n$-vertex graph. If there exists a $\mathcal{O}(1/n)$-vertex-spread measure on the set $\Emb(H,G)$ of embeddings of $H$ into $G$, then there exists a $\mathcal{O}(n^{-1/m_1(H)})$-spread measure on the set $\mathcal{F}_H$ of all unlabelled copies of $H$ in $G$.
\end{theorem} 

\begin{proof}
    Let~$H$ and~$G$ be given as above, and let~$\mu$ be a given $(C/n)$-vertex-spread measure on the set of the embeddings of~$H$ into~$G$, for some constant $C > 0$. We may assume that $C > 1$ as a $q$-vertex-spread measure is also $q'$-vertex-spread for $q' \geq q$.  Based on the intuitive relation between the two notions of spreadness, we define a probability measure $\mu'$ on $\mathcal{F}_H$ as follows. 
    $$\mu'\!\bigl(F\bigr) \coloneqq  \mu\bigl(\{ \varphi \in \Emb(H,G) : \varphi(E(H)) = F\, \}\bigr), \text{ for all $F \subseteq E(G)$.}$$
    \indent We show that $\mu'$ is the required $q$-spread measure for $q = C'n^{-1/m_1(H)}$, where $C' = (C^2\Delta)^{1/m_1(H)}$. For this, let~$S \subseteq E(G)$ be given. We need to show that the measure of all copies of~$H$ containing~$S$, given by $\mu'\bigl(\langle S\rangle\bigr) \coloneqq \mu'\bigl(\{F \in \mathcal{F}_H : S \subseteq F\}\bigr)$ is at most $q^{|S|}$. We can assume that~$S$ sits in some copy of $H$ in $G$, for otherwise $\mu'\bigl(\langle S \rangle\bigr) = 0$ and the spread condition holds trivially.  
    
    Let $V(S)$ be the set of vertices that are incident to some edge in $S$. Let $|V(S)| = m$ and let~$S$ form $c$ connected components as a subgraph of $G$. Then, for any $m$-vertex subset $X \subseteq V(H)$, and any bijection $\psi : V(S) \to X$, we define $(X, \psi)$ to be a \emph{useful pair}, if $\psi(S) \subseteq E\bigl(H[X]\bigr)$ --- that is if~$\psi$ maps all elements of~$S$ to some edge in~$H[X]$. 
    
    Observe that if the map $\varphi \in \Emb(H,G)$ induces a copy of~$H$ in~$G$ such that we have $S \subseteq \varphi\bigl(E(H)\bigr)$, then $\varphi$ can be written as an \emph{extension} of a map~$\psi^{-1}$, for some useful $(X, \psi)$ pair: Indeed, given such a map~$\varphi$, take $X = \varphi^{-1}\bigl(V(S)\bigr)$ and $\psi = \varphi^{-1}\big|_{V(S)} : V(S) \to X$. Now, if~$uv$ is an edge in~$S$, then we have that $\varphi^{-1}(uv) \subseteq E(H)$ by our assumption $S \subseteq \varphi\bigl(E(H)\bigr)$, and $\varphi^{-1}(u), \varphi^{-1}(v) \in X$ by our choice of~$X$. Thus, it follows that $\psi(S) = \varphi^{-1}(S) \subseteq E\bigl(H[X]\bigr)$, thereby making $(X, \psi)$ a useful pair as desired. We now have the following upper bound.
    \begin{equation*}
        \begin{split}
            \mu'\bigl(\langle S \rangle\bigr) &= \mu \Bigl(\bigl\{ \varphi \in \Emb(H,G) : S \subseteq \varphi(E(H)) \bigr\}\Bigr)\\
            &\leq \sum_{(X, \psi) \text{ useful}} \mu\Bigl( \bigl\{ \varphi \in \Emb(H,G) : \varphi(x) = \psi^{-1}(x) \text{ for all } x \in X \bigr\}\Bigr)\\ 
            &\leq \sum_{(X, \psi) \text{ useful}} \biggl(\frac{C}{n}\biggr)^{\!m},
        \end{split}
    \end{equation*}
    where the last inequality uses the vertex-spread property for $\mu$ on the set of embeddings for which the images of $m$~vertices have been fixed by the map $\psi^{-1}$. 
    
    Next, we provide an upper bound on the number of useful $(X, \psi)$ pairs by constructing them as follows: First, for each component of~$S$, pick a special root vertex and map it to some distinct vertex in~$V(H)$. This can be done for each of the~$c$ components of~$S$ in at most~$n^c$ ways. Now, iteratively pick an unmapped vertex~$u \in V(S)$ which has a mapped neighbour~$v \in V(S)$. Then, as~$uv$ is an edge in~$S$, by the definition of a useful pair, $\psi(u)\psi(v)$ must form an edge in $H$. As $\Delta(H) \leq \Delta$, there are at most~$\Delta$ options to choose the image~$\psi(u)$ in~$V(H)$. Repeat this until all remaining $m-c$ vertices of $S$ get mapped into~$V(H)$. Thus, with possible over-counting, there are at most $n^c \cdot \Delta^{m-c}$ useful $(X, \psi)$ pairs for $S$. Combining this estimate with the above inequality, we have that  
    $$\mu'\bigl(\langle S \rangle\bigr) \leq n^c \Delta^{m-c}\cdot \bigl(C/n\bigr)^m =  \bigl({C^2\bigr)^{m/2}\bigl(\Delta}/{n}\bigr)^{m-c} \leq \bigl(C^2\Delta/n\bigr)^{m-c},$$
    where in the last inequality, we used the fact that as every vertex in $V(S)$ is adjacent to an edge in $S$, the number of components~$c$ formed by $S$ is at most $m/2$ and hence, $m - c \geq m/2$.
    
    Finally, for each $i \in [c]$, let~$S_i$ be the set of edges of~$S$ in its~$i^{th}$ connected component. Then, the definition of~$m_1(H)$ implies that $|S_i| \leq m_1(H)\cdot(\, |V(S_i)| -1)$. Adding this inequality for each of the~$c$ components of $S$, we get the bound $|S| \leq m_1(H)(m-c)$. Combining this with the above inequality, and for $C' = (C^2\Delta)^{1/m_1(H)}$ and sufficiently large $n$, the measure $\mu'$ is shown to be a $C'n^{-1/m_1(H)}$-spread measure as follows:
    \[ \mu'\bigl(\langle S \rangle \bigr) \leq \big({C^2\Delta}/{n}\big)^{{|S|}/{m_1(H)}} = \big(C'n^{-1/m_1(H)}\big)^{|S|}\,. \qedhere \]
\end{proof}

\subsection{A Vertex-spread Blow-up Lemma}

By Theorem~\ref{thm:vtx_spread_implies_spread}, it follows that in order to apply Theorem~\ref{thm:FKNP_spread_measure}, it suffices to construct a $\mathcal{O}(1/n)$-vertex-spread measure on the set $\Emb(H,G)$.  We show that if the graphs~$G$ and~$H$ are given with reduced graphs, partitions, buffer vertices, and image restrictions that satisfy the conditions of the blow-up lemma (Lemma~\ref{lem:dense_blowup}), then there exists such a $\mathcal{O}(1/n)$-vertex-spread measure on $\Emb(H,G)$. This vertex-spread version of the blow-up lemma is stated as follows. 

\begin{theorem}[Spread Blow-up Lemma] \label{thm:spread_dense_blowup}
  For all $\Delta\ge 2$, $\Delta_{R'} \in \mathbb{N}$, $\alpha, \zeta, d>0$, and ${\kappa>1}$, there exists $\eps, \rho > 0$ such that for all $r_1 > 0$, there exists a constant $C_S > 0$ such that the following holds: 
  
  Let $R$ be a graph on $r\le r_1$ vertices and let $R'\subseteq R$ be a spanning subgraph with $\Delta(R')\leq \Delta_{R'}$. Let $H$ and $G$ be $n$-vertex graphs, given with $\kappa$-balanced, size-compatible vertex partitions $\mathcal{X}=\{X_i\}_{i\in[r]}$ and $\mathcal{V}=\{V_i\}_{i\in[r]}$, respectively, which have parts of size at
  least $n/(\kappa r_1)$. Let 
  $\tilde{\mathcal{X}}=\{\tilde{X}_i\}_{i\in[r]}$ be a family of subsets of $V(H)$, and $\mathcal{I} = \{I_x\}_{x \in V(H)}$ be a family of subsets of $V(G)$.  
  Suppose that
  \begin{enumerate}[label=\itmarab{BL}]
  \item \label{itm:SpBL1} the partition $(G,\mathcal{V})$ is $(\eps,d)$-regular on $R$ and $(\eps,d)$-super-regular on $R'$,
  \item \label{itm:SpBL2} $\Delta(H)\leq \Delta$ and $(H,\mathcal{X})$ is an $R$-partition,
  \item \label{itm:SpBL3} $\tilde{\mathcal{X}}$ is an $(\alpha,R')$-buffer for $(H,\mathcal{X})$, and
  \item \label{itm:SpBL4} $\mathcal{I} = \{I_x\}_{x \in V(H)}$ is a family of $(\rho, \zeta)$-image-restrictions. 
  \end{enumerate}

  Then there exists a $(C_S/n)$-vertex-spread probability measure on the set of all embeddings ${\varphi: V(H) \to V(G)}$ of $H$ into $G$, for which $\varphi(x) \in I_x$ for all $x \in V(H)$.
\end{theorem}

Note that the setting and the sufficient conditions required by Theorem~\ref{thm:spread_dense_blowup} are identical to that of the blow-up lemma (Lemma~\ref{lem:dense_blowup}). The only difference is that while Lemma~\ref{lem:dense_blowup} returns a single copy of the graph~$H$ in~$G$, Theorem~\ref{thm:spread_dense_blowup} returns a vertex-spread measure on the set of all embeddings of~$H$ into~$G$. Moreover, this vertex-spread measure is supported on those embeddings in $\Emb(H, G)$ that map vertices of $X_i$ into $V_i$ for all $i \in [r]$ and satisfy the given image restrictions. The proof of Theorem~\ref{thm:spread_dense_blowup} results from a careful
probabilistic analysis of the key steps used in the proof of Lemma~\ref{lem:dense_blowup} and is deferred to Section~\ref{sec:spreadblow}.  

\section{Proof of the Main Technical Theorem}
\label{sec:robust_overview}

In this section, we prove Theorem~\ref{thm:Robust_main}. Let us first make a quick comment on the case when $\Delta(H) \leq 1$. When $\Delta = 1$, the extension threshold~$\delta_{\rm e}(1)$ equals $1/2$ since $\Delta/(\Delta+1) \leq \exth \leq (2\Delta -1)/2\Delta$. Further, every graph~$H$ with $\Delta(H) \leq 1$ is a vertex disjoint union of a matching and a collection of isolated vertices, and hence is contained in any Hamilton cycle of $G$. Thus, Theorem~\ref{thm:Robust_main} for $\Delta = 1$ follows directly from the following theorem on robustness of Hamiltonicity for Dirac Graphs~\cite{robust_hamiltonicity}. Note that $m_1(H) = 1$ for a non-empty graph~$H$ with $\Delta(H) \leq 1$. 

\begin{theorem}[Robust Hamiltonicity~\cite{robust_hamiltonicity}]
\label{thm:Hamilton_robust}
There exists a constant~$C_1 > 0$ such that if~$G$ is an $n$-vertex graph with minimum degree~$\delta(G) \geq n/2$, then the random subgraph~$G(p)$ is Hamiltonian with high probability for $p \geq C_1 n^{-1} \log n$.
\end{theorem}

We can thus assume that $\Delta(H) \geq 2$. From Section~\ref{sec:robust}, we see that it suffices to show the existence of a $\mathcal{O}(1/n)$-vertex-spread measure on the set $\Emb(H,G)$ using the spread blow-up lemma in order to prove Theorem~\ref{thm:Robust_main}. For indeed, then Theorem~\ref{thm:Robust_main} follows directly from an application of Theorems~\ref{thm:FKNP_spread_measure} and \ref{thm:vtx_spread_implies_spread}. We prove Theorem~\ref{thm:Robust_main} formally towards the end of this section. 

In order to be able to apply Theorem~\ref{thm:spread_dense_blowup} on the graphs~$G$ and~$H$, as given in the setting of Theorem~\ref{thm:Robust_main}, we need to partition and process both graphs in a manner that the sufficient conditions~\ref{itm:SpBL1}--\ref{itm:SpBL4} of Theorem~\ref{thm:spread_dense_blowup} are satisfied. For this application of Theorem~\ref{thm:spread_dense_blowup}, we do not need to impose any image restrictions and so~\ref{itm:SpBL4} will be vacuously satisfied.

Given the graph~$G$, we first fix our underlying reduced graphs~$R$ and~$R'$ on, say, $r$ vertices, and construct a partition~$\mathcal{V} = \{V_i\}_{i \in [r]}$ of $V(G)$ in a manner that satisfies~\ref{itm:SpBL1}. This is done via the following lemma for the graph~$G$, which we prove later in Section~\ref{sec:robust_G_lemma}.

\begin{lemma}[Lemma for $G$]
    \label{lem:G_lemma}
    For all $\Delta \geq 2$ and $\gamma > 0$, there exist $\ d_0 > 0$ and $\kappa > 1$, such that for all $\eps > 0$ and $0 < d < d_0$, there exists $r_0 \in \mathbb{N}$ for which the following holds.
    
    For every graph~$G$ with $\delta(G) \geq (\exth + \gamma )n$, there exist a spanning subgraph~$G' \subseteq G$, a reduced graph~$R$ on $r \leq r_0$ vertices, where $\Delta+1$ divides $r$; a spanning $K_{\Delta+1}$-factor~$R'$ of~$R$; and a partition $\mathcal{V} = \{V_1, V_2, \dots, V_{r}\}$ of $V(G') = V(G)$, with the following properties.
    \begin{enumerate}[label=\itmarab{G}]
    \item \label{itm:G1} The partition~$(G',\mathcal{V})$ is a $\kappa$-balanced partition of~$V(G')$ with $$\frac{n}{\kappa r} \leq \Bigl(1 - \frac{\gamma}{8}\Bigr) \frac{n}{r} \leq |V_i| \leq \frac{\kappa n}{r} \text{ for all $i \in [r]$, }$$
    \item \label{itm:G2}the reduced graph~$R$ has minimum degree $\delta(R) \geq (\exth+ \gamma/4)r$, and
    \item \label{itm:G3} the partition~$(G', \mathcal{V})$ is $(\eps, d)$-regular on~$R$ and $(\eps, d)$-super-regular on~$R'$.
    \end{enumerate}
    \end{lemma}

 Next, we work with the graph~$H$. Given the partition $(G, \mathcal{V})$ of $V(G)$ and the reduced graphs~$R$ and~$R'$, we will construct a partition $\mathcal{X} = \{X_i\}_{i \in [r]}$ of $V(H)$, such that $(H, \mathcal{X})$ is an $R$-partition and is size-compatible with~$(G,\mathcal{V})$. This  satisfies~\ref{itm:SpBL2}. At this stage, while constructing~$(H, \mathcal{X})$, we also define and set aside potential buffer vertices $\tilde{X}_i \subseteq X_i$ in a manner that satisfies \ref{itm:SpBL3}. This is done in the following lemma for~$H$, which we prove later in Section~\ref{sec:robust_H_lemma}.

 \begin{lemma}[Lemma for $H$]
    \label{lem:Lemma_H}
    For all $\Delta \geq 2$, $\gamma > 0$, and $\kappa > 1$, there exists a constant $\alpha > 0$, such that for all $\ell \geq \Delta+1$, the following holds. 
    
    Given an $n$-vertex graph~$G$, let~$R$ be an $r$-vertex reduced graph for~$G$ with ${\delta(R) \geq (\exth+ \gamma)r}$ and with a $K_{\ell}$-factor~$R'$ of~$R$. Let $\mathcal{V} = \{V_i\}_{i \in [r]}$ be a partition of $V(G)$ with parts of size $(1 -\gamma/2 ){n}/{r} \leq |V_i| \leq \kappa n/r$ for all $i \in [r]$. Let $H$ be an $n$-vertex graph with $\Delta(H) \leq \Delta$.
    
    Further, suppose that $\mathcal{X}^* = \{X^*_i\}_{i \in [r]}$ is a collection of pairwise disjoint subsets  of $V(H)$ of size $|X^*_i| \leq \alpha |V_i|$ and such that $\bigl(H[X^*], \mathcal{X}^*\bigr)$ is an $R$-partition, where $X^* \coloneqq \bigcup_{i \in [r]} X_i^*$. Then there exists a partition $\mathcal{X} = \{X_{i}\}_{i \in [r]}$ of $V(H)$, and subsets $\tilde{X}_i \subseteq X_i$ with the following properties. 
    \begin{enumerate}[label=\itmarab{H}]
    \item \label{itm:H1} The partition~$(H, \mathcal{X})$ is size-compatible with the partition~$(G,\mathcal{V})$,
    \item \label{itm:H2} the partition $(H, \mathcal{X})$ is an $R$-partition, 
    \item \label{itm:H3} the family of subsets~$\tilde{\mathcal{X}} = \{\tilde{X}_i\}_{i \in [r]}$ is an $(\alpha, R')$-buffer for $(H, \mathcal{X})$, and
    \item \label{itm:H4} for all $i \in [r]$, we have $X_i^* \subseteq X_i$ and $X_i^* \cap \tilde{X}_i = \emptyset$.
    \end{enumerate}
\end{lemma}

We wish to make two quick remarks here. First, the presence of the small $R$-partition $\mathcal{X}^*$ in the lemma above, while not required in the proof of Theorem~\ref{thm:Robust_main}, has been added for its utility in dealing with image restrictions which for instance, has been used in our companion result on the local resilience for containment of bounded degree graphs \cite{ABKN_resil}. 

Secondly, and more significantly, the proof of Lemma~\ref{lem:Lemma_H} is the only part in the entire proof of Theorem~\ref{thm:Robust_main} where the extension property of the extension threshold~$\exth$ is used. This application arises from the need to obtain a size-compatible $R$-partition of~$V(H)$, which builds on a good partial embedding of carefully chosen potential buffer vertices. 

We are now ready to prove Theorem~\ref{thm:Robust_main}, modulo Lemmas~\ref{lem:G_lemma} and~\ref{lem:Lemma_H} stated above.

\begin{proof}[Proof of Theorem \ref{thm:Robust_main}]
    Let $\Delta \in \mathbb{N}$ and $\gamma > 0$ be given, and let $G$ and $H$ be two $n$-vertex graphs that satisfy the conditions of Theorem~\ref{thm:Robust_main}. If $\Delta = 1$, then Theorem~\ref{thm:Robust_main} follows directly from Theorem~\ref{thm:Hamilton_robust} with $C = C_1$. Thus, we may assume that $\Delta \geq 2$. We may further assume that the graph~$H$ has at least~$n/4$ edges, since otherwise, the graph~$H$ has at least~$n/2$ isolated vertices and hence, adding a matching of~$n/4$ edges on this set of isolated vertices does not alter the value of~$m_1(H)$.    
    
    Let $d_0 > 0$ and $\kappa >1$ be constants returned by the Lemma for~$G$ (Lemma~\ref{lem:G_lemma}) with~$\Delta$ and ~$\gamma$ as input. 
    Next, with input~$\Delta$,~$\gamma/4$, and~$\kappa$, let~$\alpha > 0$ be the constant returned by the Lemma for~$H$ (Lemma~\ref{lem:Lemma_H}). 
    Then, with input $\Delta$, $\alpha$, $\kappa$ as above, and with $d = d_0/2$, $\Delta_{R'} = \Delta$, and $\zeta = 1$, the spread blow-up lemma (Theorem~\ref{thm:spread_dense_blowup}) returns constants $\eps, \rho > 0$. 
    
    With input $\eps$ and $d$, Lemma~\ref{lem:G_lemma} returns a constant~$r_0 \in \mathbb{N}$ such that for the given graph $G$, there exist a spanning subgraph $G' \subseteq G$, reduced graphs~$R$ and~$R'$ with $r \leq r_0$ vertices, and a partition $\mathcal{V} = \{V_i\}_{i \in [r]}$ of $V(G')$ such that \ref{itm:G1}--\ref{itm:G3} hold. With this choice of $r_0$, Theorem~\ref{thm:spread_dense_blowup} also returns a constant $C_S > 0$ for which the conclusion of Theorem~\ref{thm:spread_dense_blowup} holds.
    
    By Lemma~\ref{lem:G_lemma}, $R'$ is a  $K_{\Delta+1}$-factor of the $r$-vertex reduced graph~$R$. Moreover, by \ref{itm:G1} and \ref{itm:G2}, the graph~$R$ has minimum degree $\delta(R) \geq (\exth+\gamma/4)r$ and each part~$V_i$ of the partition~$\mathcal{V}$ has size in the range $(1 - \gamma/8)n/r \leq |V_i| \leq \kappa n/r$. 
    Thus, for the graph~$H$, an application of Lemma~\ref{lem:Lemma_H}, with $\ell = \Delta+1$ and $X_i^* = \emptyset$ for all $i \in [r]$, returns a partition~$\mathcal{X} = \{X_i\}_{i\in [r]}$ of $V(H)$ and subsets $\tilde{X}_i$ of $X_i$, such that they satisfy~\ref{itm:H1}--\ref{itm:H4}. Finally set $I_x = V_i$ for all $x\in X_i$ and for all $i \in [r]$.  

    We are now in the setting of Theorem~\ref{thm:spread_dense_blowup}: Clearly,~$R'$ is a spanning subgraph of~$R$ with $\Delta(R') = \Delta = \Delta_{R'}$. By~\ref{itm:G1} and~\ref{itm:H1}, $\mathcal{V}$ and $\mathcal{X}$ are size-compatible $\kappa$-balanced partitions of $V(G')$ and $V(H)$ respectively, with each part of size at least $n/\kappa r \geq n/\kappa r_0$. Further, all conditions of Theorem~\ref{thm:spread_dense_blowup} are met. Indeed, \ref{itm:SpBL1}--\ref{itm:SpBL3} follow directly from~\ref{itm:G3}, \ref{itm:H2}, and \ref{itm:H3}, respectively; and \ref{itm:SpBL4} holds vacuously as there are no image-restricted vertices in $H$. Hence, by an application of Theorem~\ref{thm:spread_dense_blowup}, there exists a $(C_S/n)$-vertex-spread measure on the set $\Emb(H,G)$.

    Set~$X = E(G)$ and let~$\mathcal{F}_H$ be the family of subsets of~$E(G)$ that form unlabelled copies of~$H$ in~$G$. By Theorem~\ref{thm:vtx_spread_implies_spread}, the $(C_S/n)$-vertex-spread measure constructed above implies the existence of a $q$-spread measure on $\mathcal{F}_H$ with $q \leq C^*n^{-1/m_1(H)}$ for some constant~$C^*$. Let~$K$ be the universal constant from Theorem~\ref{thm:FKNP_spread_measure} and define $C \coloneqq 2KC^*$ to be the constant required by Theorem~\ref{thm:Robust_main}. Let $p \geq Cn^{-1/m_1(H)}\log n$ be given. Note that for the family~$\mathcal{F}_H$, the maximum size of a minimal set~in $\mathcal{F}_H$ is bounded above by $\ell(\mathcal{F}_H) \leq |E(H)| \leq n^2$. Thus, as $q \leq C^*n^{-1/m_1(H)}$ and as $\log\ell(\mathcal{F}_H) \leq 2\log(n)$, we have that $$p \geq Cn^{-1/m_1(H)}\log n \geq K \bigl(C^*n^{-1/m_1(H)}\bigr)\cdot 2\log n \geq K q \log\ell(\mathcal{F}_H).$$
    As $n/4 \leq  E(H)\leq \ell(\mathcal{F}_H) \leq n^2$, it follows that~$\ell(\mathcal{F}_H)$ tends to infinity with~$n$. Thus, by Theorem~\ref{thm:FKNP_spread_measure}, the random subgraph~$G(p)$ contains some copy of the graph~$H$ with probability tending to~$1$ as $\ell(\mathcal{F}_H)$ (and hence~$n$) tends to infinity, as required. 
\end{proof}

\section{Proof of the Spread Blow-up Lemma}
\label{sec:spreadblow}

In this section, we prove our vertex-spread version of the blow-up lemma (Theorem~\ref{thm:spread_dense_blowup}).

As mentioned earlier, the blow-up lemma from \cite{sparse_blowup} is a special case of a more general version for sparse graphs. The authors of \cite{sparse_blowup} proved three separate versions of the sparse blow-up lemma, each of which is applicable in different settings, such as for embedding into random graphs or $(p, \beta)$-bijumbled pseudorandom graphs. Moreover, it is possible to recover Lemma~\ref{lem:dense_blowup} from each of these versions by setting $p = 1$ in their respective statements. 

The proof of Theorem~\ref{thm:spread_dense_blowup}, which is deferred to the end of this section, relies heavily on the proof of these sparse blow-up lemmas applied with $p = 1$. The random greedy algorithm and the technique of embedding buffer vertices in those proofs are useful to obtain the required vertex-spread measure, and we will explain in this section why this is the case. 
For our proof of Theorem~\ref{thm:spread_dense_blowup}, we shall use the version of the sparse blow-up lemma for $(p, \beta)$-bijumbled pseudorandom graphs~\mcite[Lemma 1.25]{sparse_blowup}. Using the proof for this version (as opposed to the other two  versions) has primarily two advantages: The random greedy algorithm embeds vertices of $H$ into $G$ in a fixed order, and the buffer vertices set aside to be used in the second stage of the proof remain constant and unaffected during the first stage of the embedding (the random greedy algorithm).

We have endeavoured to keep this section as self-contained as possible by outlining all the required details and ideas from the proof of the blow-up lemma, Lemma~\ref{lem:dense_blowup} (referred to henceforth as `the proof'). However, we shall keep this brief and the reader is welcome to refer to~\cite{sparse_blowup} for  further explanations: The introductory definitions and the proof of the sparse blow-up lemma for bijumbled graphs can be found in \mcite[Chapter 1]{sparse_blowup} and \mcite[Chapter 4]{sparse_blowup}, respectively. We shall start by outlining the proof of Lemma~\ref{lem:dense_blowup}. 

Henceforth, we shall assume that all sufficient conditions for Lemma~\ref{lem:dense_blowup} (and hence for Theorem~\ref{thm:spread_dense_blowup}) are satisfied for some given graphs~$G$ and $H$, reduced graphs~$R$ and $R'$, partitions~$(G, \mathcal{V})$ and $(H, \mathcal{X})$, and family of subsets~$\tilde{\mathcal{X}} = \{\tilde{X}_i\}_{i \in [r]}$ and $\mathcal{I} = \{I_x\}_{x \in V(H)}$. 
The proof begins by pre-processing both the input graphs~$G$ and~$H$. 
The vertices of~$H$ are partitioned into two parts, namely the \emph{main} and the \emph{buffer} vertices, denoted by~$X^{\main}$ and~$X^{\buf}$, respectively. 
The buffer vertices are equitably and uniformly chosen from the collection of potential buffer vertices $\tilde{\mathcal{X}} = \{\tilde{X}_i\}$, 
such that $|X^{\buf} \cap \tilde{X}_i| = \mu|X_i|$, for some constant $\mu$ fixed during the proof. On the other hand, the vertices of~$X^{\main}$ are assigned an ordering, which stays fixed during the entire embedding algorithm, 
and defines the order in which vertices of~$X^{\main}$ are embedded into~$G$. 

The embedding is done in a manner so that the vertices of $X_i$ are embedded only to vertices in~$V_i$ (more specifically to $I_x \subseteq V_i$). To simplify discussion, we use the notation $V(x) \coloneqq V_i$ for all $x \in X_i$ and for all $i \in [r]$.
Now, the proof embeds the vertices of~$V(H)$ into~$V(G)$ in two steps: First, the vertices in~$X^{\main}$ are embedded into~$V(G)$ one-by-one according to the fixed ordering using a \emph{random greedy algorithm} (RGA). Then, in the second step, the proof considers an auxiliary bipartite graph where each $x \in X^{\buf}$ is made adjacent to all of its \emph{available candidate} vertices in $V(G)$. This auxiliary bipartite graph can be shown to satisfy Hall's matching theorem, and the obtained perfect matching is used to complete the embedding. 

\subsubsection*{The First Step: The Random Greedy Algorithm (RGA)} 

Let $x_1, x_2, \dots, x_m$ be the fixed ordering on $X^{\main}$ mentioned above, and let $\varphi_t$ denote the partial embedding of~$H$ into~$G$ constructed until time~$t$ that maps the vertices $x_1, \dots, x_t$ into~$V(G)$ such that $H\bigl[\{x_1, \dots, x_t\}\bigr]$ is a subgraph of $G\bigl[\im(\varphi_t)\bigr]$. 
Further, for $i \in [m]$, let $N^{<}_H(x_i)$ denote the neighbours of~$x_i$ in~$H$ which precede~$x_i$ in the ordering. 
Then at any time~$t \leq m$, the algorithm embeds the vertex $x_t$ to some vertex~$v \in V(x_t) \cap I_{x_t}$, which is available to be embedded to and is a common neighbour of the vertices $\varphi_{t-1}(N_H^{<}(x_t))$. If this set of \emph{available candidate vertices} is too small, then the algorithm halts with failure. 

The proof of the blow-up lemma shows that halting with failure is unlikely. Roughly speaking, this holds true as the graph~$H$ provides limited restriction to candidate vertices due to its bounded degree, and as the somewhat large size of the set-aside buffer vertices $|X^{\buf}| = \mu|V(H)|$ allows for the set of available vertices to stay large till time $t = m$. True to its name, at each time step $t \leq m$, the \emph{random greedy algorithm (RGA)} simply picks an image for~$x_t$ from its large set of available candidate vertices, uniformly at random, to extend the partial embedding~$\varphi_{t-1}$ to~$\varphi_t$. More precisely, the following was shown to be true regarding the RGA \mcite[Claim 4.4]{sparse_blowup} (rephrased for $p=1$).
    
\begin{enumerate}[start = 3, label = \itmarab{INV}]
    \item \label{itm:INV3_old} At each time $t$ in the running of the Random Greedy Algorithm, when we embed $x_t$ to create $\varphi_t$, we do so uniformly at random into a set of size at least $\frac{1}{10}\mu \zeta |V(x_t)|$.
\end{enumerate}
Here, $\mu$ and $\zeta$ are some small positive constants which arise in the proof of the blow-up lemma. Note that by assumption on the sizes of the parts in the partition $\mathcal{V}$ in Lemma \ref{lem:dense_blowup}, we have that $|V(x)| \geq |V(G)|/\kappa r_1$. Hence, \ref{itm:INV3_old} can be restated as follows. 

\begin{enumerate}[start = 1, label = \itmarab{RGA}]
    \item \label{itm:INV3} At each time $t \leq m$, the random greedy algorithms embeds $x_t$ uniformly at random into a set of size at least $2\nu n $, where $\nu = \mu \zeta / 20 \kappa r_1$.
\end{enumerate}

We shall use this property to show that the RGA embeds vertices of $X^{\main}$ into $V(G)$ in a $(1/\nu n)$-vertex-spread manner. More precisely, the RGA produces a distribution on $\Emb(H[X^\main],G) \cup \{\fail\}$, where $\{\fail\}$ denotes the event that the RGA was not able to embed all vertices of $X^{\main}$ and halted with failure, or that the resulting partial embedding~$\varphi_m$ does not have the desirable properties~\ref{itm:GPE3},~\ref{itm:PRGA2}, and~\ref{itm:PRGA3} as introduced in the following subsection. These properties of $\varphi_m$ are useful during the second stage of embedding buffer vertices, and we include them in the event $\{\fail\}$ for convenience. 

It is shown in the proof of~\cite[Lemma~4.1]{sparse_blowup} that the probability of $\{\fail\}$ tends to zero as~$n$ tends to infinity. In particular, this means that with probability at least~$1/2$, the RGA produces an embedding of~$H[X^{\main}]$ into~$G$ with the property~\ref{itm:INV3}, along with the other desirable properties~\ref{itm:GPE3},~\ref{itm:PRGA2}, and~\ref{itm:PRGA3}. We now argue that conditioning on this success of the RGA gives us a vertex-spread distribution on the partial embeddings as required. 

\begin{lemma}
    \label{lem:v_main_vtx_spread}
    For any $k \leq m$, let $x_{t_1}, x_{t_2}, \dots, x_{t_k}$ be $k$ distinct vertices of $X^{\main}$ with $0 = t_0 < t_1 < \dots < t_k \leq m$, and let $v_{t_1}, v_{t_2}, \dots, v_{t_k}$ be distinct vertices in $V(G)$. Then, conditioned on the success of the RGA, the random embedding $\varphi_m$ of $H[X^{\main}]$ into $G$ generated by the RGA is such that 
    \[ \mathds{P}\Bigl(\bigl\{ \varphi_m \in \Emb(H[X^{\main}], G): \varphi_m (x_{t_i}) = v_{t_i} \text{ for all } i \in [k]\bigr\}\Bigr) \leq \Bigl( 1/{\nu n}\Bigr)^k.\]
\end{lemma}
\begin{proof}
    Suppose the random greedy algorithm successfully generates the partial embedding~$\varphi_m$ of~$H[X^{\main}]$ into~$G$, and let $x_{t_1}, x_{t_2}, \dots, x_{t_k} \in X^{\main}$ and $v_{t_1}, v_{t_2}, \dots, v_{t_k} \in V(G)$ be given. 
    \begin{claim}
    \label{clm:RGA_vtx_spread}
        For the random partial embedding~$\varphi_{t_k}$ up to time~$t_k$, we have 
        \[ \mathds{P}\Bigl(\bigl\{ \varphi_{t_k}\bigm| \varphi_{t_k}(x_{t_i}) = v_{t_i} \text{ for all } i \in [k]\bigr\}\Bigr) \leq \Bigl( 1/2\nu n \Bigr)^k.\]
    \end{claim}

    \begin{claimproof}
    We prove the claim by induction on~$k$. The claim holds trivially for $k = 0$. Now, suppose the statement is true for $k-1$.  
    Define~$\mathcal{A}_i$ for $i \in [k]$ to be the event that the random embedding $\varphi_{t_i}$ maps $x_{t_j} \mapsto v_{t_j}$ for all $j \in [i]$. 
    By the induction hypothesis, then, we have $\mathds{P}(\mathcal{A}_{k-1}) \leq (1/2\nu n)^{k-1}$. 
    Define $H_k = (v_{t_{k-1}+1}, \dots, v_{t_k-1})$ to be the random vector of vertices in $V(G)$ to which the vertices $x_{t_{k-1}+1}, \dots, x_{t_k-1}$ are embedded by the RGA. Then, as $\mathcal{A}_k \subseteq \mathcal{A}_{k-1}$, we have
    \begin{eqnarray*}
        \mathds{P}\bigl(\mathcal{A}_k\bigr) &=& 
        \smashoperator{\sum_{h_k \in\, \supp(H_k)}}
        \;\mathds{P}\bigl(\varphi_{t_k}(x_{t_k}) = v_{t_k} \bigm| H_k = h_k;  \mathcal{A}_{k-1}\bigr)
        \cdot\mathds{P}\bigl(H_k = h_k \bigm| \mathcal{A}_{k-1}\bigr)\cdot \mathds{P}\bigl(\mathcal{A}_{k-1}\bigr)\\
        &\leq& \frac{1}{2\nu n}\cdot \biggl(\frac{1}{2\nu n}\biggr)^{k-1} \cdot \smashoperator{\sum_{h_k \in \, \supp(H_k)}} \; \mathds{P}\bigl(H_k = h_k \bigm| \mathcal{A}_{k-1}\bigr) = \biggl(\frac{1}{2\nu n}\biggr)^k,
    \end{eqnarray*}
    where the sum is taken over the support of the random vector $H_k$, that is, over all realisations~$h_k$ of~$H_k$ for which $\mathds{P}(H_k = h_k) >0$; and the inequality uses the fact that because of~\ref{itm:INV3}, at time~$t_k$ the RGA embeds $x_{t_k}$ to $v_{t_k}$ uniformly with probability at most $1/(2\nu n)$ if $v_{t_k}$ is an available candidate vertex for $x_{t_k}$, and with probability $0$ otherwise. 
\end{claimproof}
    
    Now, let $\mathcal{A}$ be the event that the random embedding $\varphi_m$ maps $x_{t_j} \mapsto v_{t_j}$ for all $j \in [k]$ and let $H = (v_{t_{k}+1}, \dots, v_m)$ be the random vector of vertices in $V(G)$ to which the vertices~$x_{t_k+1}, \dots, x_{t_m}$ are embedded by the RGA. We calculate $\mathds{P}(\mathcal{A})$ by an application of Claim \ref{clm:RGA_vtx_spread} as follows:
    \[\mathds{P}\bigl(\mathcal{A}\bigr) = \smashoperator{\sum_{h \in\, \supp(H)}}\;\mathds{P}\bigl(H = h \bigm| \mathcal{A}_k \bigr)\cdot \mathds{P}\bigl(\mathcal{A}_k\bigr) \leq \biggl(\frac{1}{2\nu n}\biggr)^{\!k}\cdot\; \smashoperator{\sum_{h \in \, \supp(H)}}\;\mathds{P}\bigl(H = h \bigm| \mathcal{A}_{k}\bigr) = \biggl(\frac{1}{2\nu n}\biggr)^{\!k}.\]
    \indent Finally, we would like to estimate the probability of~$\mathcal{A}$ conditioned on the RGA not failing. Since the probability of failure is at most~$1/2$, and at most all of~$\mathcal{A}$ is outside the event of failure, the conditioned probability is at most $2\mathds{P}(\mathcal{A})\le2({1}/{2\nu n})^{k}\le({1}/{\nu n})^{k}$, as required.
    \end{proof}

\subsubsection*{The Second Step: Embedding Buffer Vertices by Hall's Theorem} 

Upon successful completion of the RGA with a partial embedding~$\varphi_m$, it remains to map the buffer vertices~$X^{\buf}$ to the remaining available vertices of $V(G)$. For this section, we regard~$\varphi_m$ as a fixed embedding and show that whatever it is, we obtain the required spread distribution on mappings of the buffer vertices.

For each $i \in [r]$, let~$X_i^{\buf}$ denote the buffer vertices in the part~$X_i$, and let~$V_i^{\buf} \coloneqq V_i \setminus \im(\varphi_{m})$ be the available vertices in~$V_i$. By size-compatibility of the partitions~$\mathcal{X}$ and~$\mathcal{V}$, we have that $|X^{\buf}_i| = |V^{\buf}_i|$ for all $i \in [r]$. For each $x \in X^{\buf}_i$, let $C^{\buf}(x)$ denote the set of available candidate vertices for~$x$. Here, $C^{\buf}(x)  \coloneqq N^*_G\bigl(\varphi_m(N_H(x; X^{\main}); V_i^{\buf}\bigr)$ is defined as the set of vertices in the set $V^{\buf}_i$ that are adjacent to all images under $\varphi_m$ of the already embedded neighbours of $x$ in $X^{\main}$. As the RGA did not fail, it was shown in~\mcite[Lemma 4.1]{sparse_blowup} that the partial embedding~$\varphi_m$ generated so far is a \emph{good partial embedding}. This means that among other useful properties, the following is true (see Definition 2.24 and the first line in the proof of Claim 4.2 in~\cite{sparse_blowup}, recalling that $p=1$).

\begin{enumerate}[start = 3, label = \itmarab{GPE}]
    \item \label{itm:GPE3} For each $x \in X^{\buf}_i$, the candidate set~$C^{\buf}(x)$ has size at least $\tfrac{1}{2}\mu d^b |V_i|$.
\end{enumerate}

Here,~$b$ is a positive integer at most~$\Delta$, the constant~$d$ is the input minimum density of regular pairs, and~$\mu$ is a constant arising in the proof that denotes the fraction of buffer vertices (set aside by the proof) in $X_i$. Further, as the RGA did not fail, the following hold for the sets~$V^{\buf}$ and~$X^{\buf}$ at the end of the random greedy algorithm (rephrased for $p = 1$), as proved in~\mcite[Lemma 4.1]{sparse_blowup}. 

\begin{enumerate}[start = 2, label = \itmarab{PRGA}]
    \item\label{itm:PRGA2} Every vertex in~$V_i$ is a candidate for at least $\mu (d^{\Delta}/100)^b |X_i|$ vertices of~$X^{\buf}_i$, and
    \item\label{itm:PRGA3} for every set~$W \subseteq V_i$ of size at least~$\rho|V_i|$, there are at most~$\rho|X_i|$ vertices in $X^{\buf}_i$ with fewer than~$\tfrac{1}{2} d^b |W|$ candidates in $W$.
\end{enumerate}

Here $\rho>0$ is a constant arising in the proof, such that $\rho < 10^{-6}\mu d^b$. In order to complete the embedding of buffer vertices in a vertex-spread manner, we consider the following auxiliary bipartite graph~$F$ with vertex set~$X^{\buf} \sqcup V^{\buf}$ and such that $xv \in E(F)$ if and only if $v \in C^{\buf}(x)$. It is easy to see that the set of embeddings of~$H[X^{\buf}]$ into~$G[V^{\buf}]$ that embed~$H$ into~$G$ along with the map~$\varphi_m$ form a one-to-one correspondence with the set of perfect matchings in the bipartite graph~$F$. Thus, a~$\mathcal{O}(1/n)$-spread measure on the set of perfect matchings in~$F$ naturally gives rise to the required~$\mathcal{O}(1/n)$-vertex-spread measure on the set of embeddings of~$H[X^{\buf}]$ into~$G[V^{\buf}]$. We show this correspondence formally in the proof of Theorem~\ref{thm:spread_dense_blowup} below. 

We now prove the existence of a $\mathcal{O}(1/n)$-spread measure on the set of perfect matchings in~$F$. Note that, as for each $x \in X^{\buf}_i$ the candidate set~$C^{\buf}(x)$ is a subset of~$V_i$, the graph~$F$ is a vertex-disjoint union of bipartite graphs~$\{F_i\}_{i \in [r]}$ with vertex sets~$A_i \sqcup B_i \coloneqq X^{\buf}_i \sqcup V^{\buf}_i$. We show in the proof of Lemma~\ref{lem:buf_spread_measure} that in order to find the required spread measure on the set of perfect matchings in~$F$, it suffices to construct a $\mathcal{O}(1/n)$-spread measure on the set of perfect matchings in~$F_i$ for each of these bipartite subgraphs~$F_i$. The properties~\ref{itm:GPE3},~\ref{itm:PRGA2}, and~\ref{itm:PRGA3} can be restated for the subgraphs~$F_i$ as follows.

\begin{enumerate}[label=\itmarab{FB}]
    \item \label{itm:FB1} For each $x \in A_i$, we have $d_{F_i}(x) \geq  \frac{1}{2} d^b |B_i|$,
    \item \label{itm:FB2} for each $v \in B_i$, we have $d_{F_i}(v) \geq  (d^{\Delta}/100)^b |A_i|$, and
    \item \label{itm:FB3}for every set $W \subseteq B_i$ of size at least $\frac{\rho}{\mu}|B_i|$, there are at most $\frac{\rho}{\mu}|A_i|$ vertices in $A_i$ with fewer than $\frac{1}{2} d^b |W|$ candidates in $W$.
\end{enumerate}

In \cite{toolkit_robust_thresholds}, Pham, Sah, Sawhney, and Simkin show that if a bipartite graph on $A\sqcup B$ is $(\eps, d)$-super-regular, then there exists a $\mathcal{O}(1/n)$-spread measure on the set of perfect matchings (in fact, they prove a more general statement about star factors). Unfortunately, the graphs~$F_i$ are not necessarily super-regular and instead satisfy~\ref{itm:FB3}, which is different from the consequence of regularity used in~\cite{toolkit_robust_thresholds}. Nevertheless, we can modify their approach to work in this setting as well.

The proof begins by generating a random bipartite subgraph of $F_i$ by picking for each vertex in $V(F_i)$, a random subset of~$C$ neighbours chosen uniformly at random with replacement for some large constant~$C$, together with a random binomial subgraph of $F_i$. Let~$Z$ be the union of both these random subgraphs. We then show that with probability at least~$1/2$, the graph~$Z$ satisfies Hall's Theorem and that this implies the existence of the required spread measure on perfect matchings of $F_i$. To do this, we look at sets~$S \subseteq A_i$ and~$T\subseteq B_i$ of sizes~$k$ and~$k-1$ respectively, and show that the probability of the event~$N_Z(S) \subseteq T$ is very small. The proof then divides into two cases: when $k$ is very small or large, and when $k$ takes non-extreme values. We work with properties of the uniform random subgraph in the first case and the binomial random subgraph in the latter, both of which are unified via the graph~$Z$. In order to analyse the random subgraph~$Z$, we define~$Z$ in terms of the following coupling.

\begin{definition}
    \label{defn:coupled_measure}
    Let $F_i = (A_i \sqcup B_i, E(F_i))$ be a bipartite graph and let~$C$ be a large constant. Let~$Z_1$ be the \emph{binomial random subgraph} of~$F_i$ generated by retaining each edge of~$E(F_i)$ independently with probability~$C/\lambda$, where $\lambda = |V(F_i)|$. Similarly, let~$Z_2$ be the \emph{uniform random subgraph} of~$F_i$ obtained by choosing~$C$ neighbours uniformly at random with replacement for each~$v \in V(F_i)$. Let~$\mathds{P}_{Z_1}$ and~$\mathds{P}_{Z_2}$ denote the probability measures on~$Z_1$ and~$Z_2$ respectively. 
    
    Then, we define $Z = Z_1 \cup Z_2$ to be a random subgraph of~$F_i$ with the probability measure 
    $$ \mathds{P}_Z(Z) \coloneqq \sum_{(Z_1, Z_2)} \mathds{1}_{Z = Z_1 \cup Z_2}(Z_1, Z_2) \cdot  \mathds{P}_{Z_1}(Z_1)\, \mathds{P}_{Z_2}(Z_2),$$ where the \emph{indicator function} $\mathds{1}_A(\mathbf{x})$ equals $1$ if $A(\mathbf{x})$ is true, and $0$ otherwise.
\end{definition} 
\begin{obs}
    \label{obs:coupling_prob}
    It can be easily verified that $\mathds{P}_Z$ is a probability measure. Further, if $\mathcal{E}$ is a monotone decreasing property on graphs, then $\mathds{P}_Z(Z \in \mathcal{E}) \leq \min\big\{ \mathds{P}_{Z_1}(Z_1 \in \mathcal{E}), \mathds{P}_{Z_2}(Z_2 \in \mathcal{E})\big\}$. Indeed, if $H \subseteq G$, then $\mathds{1}_{\mathcal{E}}(G) \leq \mathds{1}_{\mathcal{E}}(H)$ as~$\mathcal{E}$ is decreasing, and hence for $i \in \{1, 2\}$, we have
    \begin{eqnarray*}
        \mathds{P}_Z(Z \in \mathcal{E}) &=& \sum_{(Z_1, Z_2)} \mathds{1}_{\mathcal{E}}(Z_1 \cup Z_2) \cdot  \mathds{P}_{Z_1}(Z_1) \mathds{P}_{Z_2}(Z_2) \leq \sum_{(Z_1, Z_2)} \mathds{1}_{\mathcal{E}}(Z_i) \cdot \mathds{P}_{Z_1}(Z_1) \mathds{P}_{Z_2}(Z_2)\\
        &\leq& \sum_{Z_i} \Bigl(\mathds{1}_{\mathcal{E}}(Z_i) \cdot \mathds{P}_{Z_i}(Z_i)\cdot \sum_{Z_{3-i}} \mathds{P}_{Z_{3-i}}(Z_{3-i})\Bigr) = \mathds{P}_{Z_i}(Z_i \in \mathcal{E}). 
    \end{eqnarray*}   
\end{obs}

With this, we are now ready to prove the existence of a $\mathcal{O}(1/n)$-spread measure on the set of perfect matchings in the auxiliary bipartite graph~$F$.

\begin{lemma}
    \label{lem:buf_spread_measure}
    Let $F$ be the auxiliary bipartite graph described above, which is a vertex-disjoint union of bipartite graphs~$\{F_i\}_{i \in [r]}$ having vertex sets~$A_i \sqcup B_i = X^{\buf}_i \sqcup V^{\buf}_i$ and satisfying properties \ref{itm:FB1}--\ref{itm:FB3}. Then, there is a constant~$C_B$ independent of $n$, where $n = |V(G)| = |V(H)|$, such that there exists a $(C_B/n)$-spread measure on the set of perfect matchings of $F$.    
\end{lemma}
\begin{proof}
    It suffices to prove the existence of a $(C_B/ n)$-spread measure on the set of perfect matchings of~$F_i$ for each $i \in [r]$. Indeed, suppose such a spread measure~$\mathds{P}_i$ exists for each~$F_i$. Let $S \subseteq E(F)$ be given with $S_i = S \cap E(F_i)$. Similarly, let~$M$ be a random perfect matching in~$F$ with $M_i = M \cap E(F_i)$. Then, as $F_i$ are vertex-disjoint bipartite graphs, the product measure $\mathds{P}_M \coloneqq \prod_{i \in [r]} \mathds{P}_i$ gives the required $(C_B/n)$-spread measure on the set of perfect matchings of $F$ as 
    \[\mathds{P}_M\bigl(\{M : S \subseteq M\}\bigr) = \prod_{i \in [r]} \mathds{P}_i\bigl(\{M_i : S_i \subseteq M_i\}\bigr) \leq \prod_{i \in [r]} \bigl( C_B/n \bigr)^{|S_i|} = \bigl( C_B/n\bigr)^{|S|}.\]
    \indent Thus, we restrict ourselves to the graph~$F_i$ for some $i \in [r]$
    having parts~$A_i$, $B_i$ of size~$\lambda$. By the assumptions on the size of~$V_i$ in the statement of Theorem~\ref{thm:spread_dense_blowup}, we have that $\lambda = |X^{\buf}_i| = \mu|X_i| \geq (\mu/\kappa r_1)n$. Let~$C$ be a large constant depending on $d$, $b$, $\kappa$, $\mu$, $\rho$, and $\Delta$, and let $(Z_1, \mathds{P}_{Z_1})$, $(Z_2, \mathds{P}_{Z_2})$, and $(Z, \mathds{P}_Z)$ be the random graphs and probability measures as in Definition~\ref{defn:coupled_measure}. 
    
    We show that the random subgraph~$Z$ satisfies Hall's theorem with probability at least~$1/2$. For this, given a positive integer~$k \in [\lambda]$, let $S \subseteq A_i$ and $T \subseteq B_i$ be vertex sets of size $|S| = k$ and $|T| = k-1$, and define~$\mathcal{A}_k$ to be the event that for some such~$S$ and~$T$, we have $N_Z(S) \subseteq T$ (that is, $S$ and $T$ witness that Hall's condition is violated). We begin by calculating the probabilities $\mathds{P}_Z(N_Z(S) \subseteq T)$ and then use them to bound the probability~$\mathds{P}_Z(\mathcal{A}_k)$ of the event $\mathcal{A}_k$ from above. 

    \noindent \textbf{Case I: $k \in \mathcal{I} \coloneqq \bigl( (2\rho/\mu)\lambda, ( 1 - \rho/\mu)\lambda\bigr)$.} For this case we rely solely on~$Z_1$. Consider the complement set $T^c = B_i\setminus T$ which has size $|T^c| > (\rho/\mu)\lambda$. Then, by~\ref{itm:FB3}, all but at most $(\rho/\mu)\lambda$ vertices of~$S$ have at least~$(d^b/2)|T^c|$ neighbours in~$T^c$ in the graph~$F_i$. Let this subset of vertices of~$S$ be denoted by~$S'$, where $|S'| \geq k - (\rho/\mu)\lambda \geq k/2$. 
    
    Hence, the number of edges in~$F_i$ between~$S'$ and~$T^c$ is at least $(k/2)\cdot(d^b/{2})|T^c| \geq (d^b\rho/4\mu) \lambda k$. Then, shifting to the binomial random graph~$Z_1$, which has edge probability $C/\lambda$, and using Observation~\ref{obs:coupling_prob} for the event $\mathcal{E} \coloneqq \{Z: N_Z(S) \subseteq T\}$, we have 
    \begin{eqnarray*}
        \mathds{P}_Z\bigl(N_Z(S) \subseteq T\bigr) &\leq& \mathds{P}_{Z_1}\!\bigl(N_{Z_1}\!(S) \subseteq T\bigr) \leq \mathds{P}_{Z_1}\!\bigl(N_{Z_1}\!(S') \subseteq T\bigr) = \mathds{P}_{Z_1}\!\bigl(E_{Z_1}\!(S', T^c) = \emptyset \bigr)\\
        &\leq& \bigl(1 - C/\lambda\bigr)^{|E_{F_i}(S', T^c)|} \leq \bigl(1 - C/\lambda\bigr)^{(d^b \rho/{4 \mu}) \lambda k} \leq e^{-C_{1}k},
    \end{eqnarray*}
    where $C_1 = (d^b \rho/4 \mu)C$. We let $C$ be large enough so that $C_1$ satisfies $e^{2-C_1}(\mu/ \rho)^2 < 1/2$. Now, by a union-bound argument and using $\binom{\lambda}{k} \leq (e\lambda/k)^k$ for large $\lambda$, we have    
    \begin{eqnarray*}
        \sum_{k \in \mathcal{I}} \mathds{P}_Z(\mathcal{A}_k) &\leq& \sum_{k \in \mathcal{I}} \binom{\lambda}{k}\binom{\lambda}{k-1} e^{-C_1k} \leq \smashoperator[r]{\sum_{k > (2\rho/\mu)\lambda}}\;\; \bigl({2e\lambda}/{k}\bigr)^{2k} e^{-C_1k}\\
        &\leq& \smashoperator{\sum_{k > (2\rho/\mu)\lambda}}\;\; \bigl(\mu^2 e^{2-C_1}/\rho^2\,\bigr)^{k} \;\leq \smashoperator[r]{\sum_{k \geq (2\rho/\mu)\lambda}}\;\; 2^{-k} \;\leq 2^{-(2\rho/\mu)\lambda} < \frac{1}{4}.
    \end{eqnarray*}
    
    \noindent \textbf{Case II: $k \leq (2\rho/\mu) \lambda$.} For this case, we rely solely on~$Z_2$. By~\ref{itm:FB1}, any vertex~$x \in S$ has degree~$d_{F_i}(x) \geq (d^b/2)\lambda$. As $|T| < k$, the probability that the event $N_{Z_2}(x) \subseteq T$ holds in~$Z_2$ is at most $\bigl(2k/d^b\lambda\bigr)^C$, which is much smaller than $1$ as $10^6\rho < d^b \mu$. Hence, for $1 \leq k \leq (2\rho/\mu)\lambda$, we have 
    \begin{eqnarray*}
        \mathds{P}_{Z}(\mathcal{A}_k) &\leq& \binom{\lambda}{k}\binom{\lambda}{k-1} \mathds{P}_Z\bigl(N_Z(S) \subseteq T\bigr) \leq \bigl(2e\lambda/{k}\bigr)^{\!2k} \,\mathds{P}_{Z_2}\bigl(N_{Z_2}(S) \subseteq T\bigr) \\
        &\leq& \biggl(\frac{2e\lambda}{k}\biggr)^{\!2k} \biggl(\frac{2k}{d^b\lambda}\biggr)^{\!Ck} = {C_2}^k\cdot \bigl(k/\lambda\bigr)^{\!(C-4)k},
    \end{eqnarray*}
    where $C_2 = 4e^2 (2/d^b)^C$. Let $C_* = C_2\cdot (2\rho/\mu)^C$, and let $C$ be large enough so that $C_* < 1$. This is possible as $4\rho < \mu d^b$. Then, letting $\mathcal{I}'$ be the interval $(\lambda^{1/3}, (2\rho/\mu)\lambda)$ and taking a union-bound over all values of $k \leq (2\rho/\mu)\lambda$, we obtain the following. 
    \begin{eqnarray*}
        \smashoperator{\sum_{k \leq (2\rho/\mu)\lambda}}\; \mathds{P}_Z(\mathcal{A}_k) &=& \sum_{k \leq \lambda^{1/3}} {C_2}^k \bigl(k/\lambda\bigr)^{k(C-4)} + \sum_{k \in \mathcal{I}'} {C_2}^k \bigl( k/\lambda\bigr)^{k(C-4)}\\
        &\leq& \sum_{k \leq \lambda^{1/3}} {\Bigl(C_2\, \lambda^{-2(C-4)/3} \Bigr)}^k + \sum_{k \in \mathcal{I}'} {\Bigl(C_2\, \bigl(2\rho/\mu\bigr)^C \Bigr)}^k \\
        &\leq& \sum_{k \leq \lambda^{1/3}} C_2\, \lambda^{-2/3} + \sum_{k \geq \lambda^{1/3}} {C_*}^k\; \leq\;  C_2\, \lambda^{-1/3} + (1- C_*)^{-1}{C_*}^{(\lambda^{1/3})} =  o_\lambda(1)
    \end{eqnarray*}
    Thus, for large values of $\lambda$, we have $\sum_{k \leq (2\rho/\mu)\lambda} \mathds{P}_Z(\mathcal{A}_k) < 1/8$. 

    \textbf{Case III: $k \geq  \big(1 - \rho/\mu\big)\lambda$.} In the random subgraph~$Z$, the event~$N_Z(S) \subseteq T$ holds if and only if we have $N_Z(T^c) \subseteq S^c$. Further, we also have $|T^c| \leq (2\rho/\mu)\lambda$ and $|S^c| = |T^c| - 1$. This resembles Case II, and so using~\ref{itm:FB2} and following the calculations as above, \emph{mutatis mutandis}, we have $\sum_{k \geq (1 - \rho/\mu)\lambda} \mathds{P}_Z(\mathcal{A}_k) < 1/8$ for large values of $\lambda$. 

    Hence, we have shown that the random graph~$Z$ fails Hall's theorem with probability at most $\sum_{k \in [\lambda]} \mathds{P}_Z(\mathcal{A}_k) \leq 1/2$, and thus with probability at least~$1/2$, the random subgraph~$Z$ contains a perfect matching. For each instance of~$Z$ which contains a perfect matching, arbitrarily fix one. This defines a random variable taking values in the set of perfect matchings of~$F_i$ together with an element~$\{\fail\}$, where $\{\fail\}$ is chosen whenever~$Z$ does not satisfy Hall's theorem. Let  $\mathds{P}_i^*$ denote the distribution of this random variable, for which the probability~$\mathds{P}_i^*\bigl(\{\fail\}\bigr)$ was shown to be at most~$1/2$. We will now show that by conditioning on not choosing the element~$\{\fail\}$, we obtain the required spread-distribution~$\mathds{P}_i$ on the set of perfect matchings of~$F_i$. 
    
    For this, let $S \subseteq E(F_i)$ be given. We can assume that~$S$ is a (not necessarily perfect) matching, for otherwise, $\mathds{P}_i\bigl(\{M : S \subseteq M\}\bigr) = 0$ and so, the spread condition is trivially satisfied. Thus, let $S$ be a matching. Observe that $\mathds{P}_i^*\bigl(\{M:S\subseteq M\}\bigr)\le\mathds{P}_Z\bigl(\{Z:S\subseteq Z\}\bigr)$ as~$M$ was chosen in~$Z$. We will bound the latter probability from above. Since~$S$ is a matching,~$Z_1$ is a binomial random subgraph, and since we selected edges of $Z_2$ independently with replacement, the events that the various edges of~$S$ are in~$Z$ are mutually independent. So let~$e \in S$ be an edge in the matching. 
    
    This edge~$e$ lies in~$Z_1$ with probability~$C/\lambda$. Similarly, the edge~$e$ lies in~$Z_2$ if either end point choses the other end point, which has probability at most $2C/d^b\lambda + (100/d^{\Delta})^b C/\lambda$. Thus for $C' = C(200/d^{\Delta})^b$, we have
    \[\mathds{P}_Z\bigl(\{Z:e\in E(Z)\}\bigr)\le \mathds{P}_{Z_1}\bigl(\{Z_1:e\in E(Z_1)\}\bigr)+\mathds{P}_{Z_2}\bigl(\{Z_2:e\in E(Z_2)\}\bigr)\le 2C'/\lambda\,.\]
    
    By independence, we have $\mathds{P}_Z\bigl(\{Z:S\subseteq Z\}\bigr)\le (2C'/\lambda)^{|S|}$, and hence this is also an upper bound on $\mathds{P}_i^*\bigl(\{M:S\subseteq M\}\bigr)$. Since the probability~$P_i^*\bigl(\{\mathrm{fail}\}\bigr)$ is at most~$1/2$, in the conditioned distribution~$\mathds{P}_i$ we have $\mathds{P}_i\bigl(\{M:S\subseteq M\}\bigr)\le 2(2C'/\lambda)^{|S|}\le (4C'/\lambda)^{|S|}$. Finally, as $\lambda = \mu |V_i| \geq \mu n/ \kappa r_1$, setting $C_B = 4\kappa r_1 C'/\mu$, it follows that $\mathds{P}_i\big(\{M: S \subseteq M\}\big) \leq (C_B/n)^{|S|}$. 
    
    Thus, for all $i \in [r]$, the probability measure~$\mathds{P}_i$ is a $(C_B/n)$-spread on the set of perfect matchings of~$F_i$, and consequently, the product measure $\mathds{P}_M \coloneqq \prod_{i \in [r]} \mathds{P}_i$ forms the required $(C_B/n)$-spread probability measure on the set of perfect matchings of~$F$.  
\end{proof}

\subsubsection*{The Proof of Theorem~\ref{thm:spread_dense_blowup}}

Over the last two subsections we have found a spread measure (from the RGA) on embeddings of~$H[X^\main]$ into~$G$ which satisfy the given image restrictions, and for each such embedding with positive probability in this measure, we constructed a spread measure on embeddings of the remaining set~$X^\buf$ to complete an embedding of~$H$ into~$G$. This entire process naturally induces a probability measure on the set of embeddings~$\Emb(H, G)$ which concur with the given image restrictions, and which we now show to be a $\mathcal{O}(1/n)$-vertex-spread probability measure. 

\begin{proof}[Proof of Theorem \ref{thm:spread_dense_blowup}]
    Suppose the sufficient conditions of Theorem~\ref{thm:spread_dense_blowup} (and hence, of Lemma~\ref{lem:dense_blowup}) hold for the given graphs~$G$ and~$H$, reduced graphs~$R$ and $R'$, partitions~$(G, \mathcal{V})$ and $(H, \mathcal{X})$, and family of subsets~$\tilde{\mathcal{X}} = \{\tilde{X}_i\}_{i \in [r]}$ and $\mathcal{I} = \{I_x\}_{x \in V(H)}$. Then, Lemmas~\ref{lem:v_main_vtx_spread} and~\ref{lem:buf_spread_measure} hold true, and as described above, the RGA and the random embedding of buffer vertices induce a probability measure, say~$\mathds{P}_S$ on the set of embeddings $\varphi: V(H) \to V(G)$ with $\varphi(x) \in I_x$. We show that~$\mathds{P}_S$ is the required $(C_S/n)$-vertex-spread measure, for some constant~$C_S$ which is defined later. 
    
    Fix $k \in [n]$ and let the vertices $x_1, x_2, \dots, x_k \in V(H)$ and $v_1, v_2, \dots, v_k \in V(G)$ be given. Without loss of generality, assume that for some~$\ell \leq k$, the vertices $x_1, x_2, \dots, x_\ell$ are in~$X^{\main}$ and the remaining $k-\ell$ vertices are in $X^{\buf}$.

    Let a partial embedding~$\varphi_{\main}$ of $H[X^{\main}]$ into $G$ be drawn from the $(1/\nu n)$-spread distribution of Lemma~\ref{lem:v_main_vtx_spread}. Let~$\mathcal{A}_{\main}$ denote the event that the generated embedding~$\varphi_{\main}$ maps~$x_i$ to~$v_i$ for all~$i \in [\ell]$. Then by Lemma~\ref{lem:v_main_vtx_spread}, we have $\mathds{P}(\mathcal{A}_{\main}) \leq (1/\nu n)^\ell$. 

    Given any~$\varphi_{\main}$ with positive probability, we define the auxiliary bipartite graph~$F$ as in the previous subsection. The perfect matchings~$M$ of~$F$ are in bijection with the embeddings $\varphi_{\buf} \in \Emb\bigl(H[X^{\buf}], G[V^{\buf}]\bigr)$ such that $\varphi_{\main}\cup\varphi_{\buf}$ is an embedding of~$H$ into~$G$ which respects the given image restrictions. Now, let~$\varphi_{\buf}$ be obtained by choosing~$M$ from the~$(C_B/n)$-spread measure of Lemma~\ref{lem:buf_spread_measure} on perfect matchings of~$F$, and let~$\mathcal{A}_{\buf}$ denote the event that~$\varphi_{\buf}$ maps~$x_i$ to~$v_i$ for all~$i \in [k]\setminus[\ell]$. This is the same as the event that~$F$ contains all of the edges~$x_iv_i$ for $i\in[k]\setminus[\ell]$, and so $\mathds{P}\bigl(\mathcal{A}_{\buf}\bigm|\varphi_{\main}\bigr) \leq (C_B/n)^{k-\ell}$. Letting~$\mathds{P}_S$ be the distribution $(\varphi_{\main},\varphi_{\buf})$ on $\Emb(H,G)$, as described above, and letting $C_S=\max(C_B,\nu^{-1})$, we have
    \begin{align*}
        \mathds{P}_S\bigl(&\{\varphi \in \Emb(H,G): \varphi(x_i) = v_i \text{ for all } i \in [k]\}\bigr) = \mathds{P}_S\bigl(\mathcal{A}_{\buf} \cap \mathcal{A}_{\main}\bigr)\\[0.5em]
        &= \;\; \smashoperator{\sum_{\varphi_{\main}\in A_\main}}\;\mathds{P}\bigl(\varphi_{\main}\bigr)\cdot \mathds{P}\bigl(\mathcal{A}_{\buf} \bigm| \varphi_{\main}\bigr) \leq \;\; \smashoperator{\sum_{\varphi_{\main}\in A_\main}}\;\mathds{P}\bigl(\varphi_{\main}\bigr)\cdot\bigl(C_B/n\bigr)^{k-\ell}\\[0.5em]
        &=\mathds{P}\bigl(\mathcal{A}_\main\bigr)\cdot \biggl(\frac{C_B}{n}\biggr)^{\!k-\ell}\le  \biggl(\frac{1}{\nu n}\biggr)^{\!\ell} \biggl(\frac{C_B}{n}\biggr)^{\!k-\ell} \leq \biggl(\frac{C_S}{n}\biggr)^{\!k}\,. \\ & \qedhere
    \end{align*} 
\end{proof}

\section{Proof of the Lemmas for \texorpdfstring{$G$}{G} and \texorpdfstring{$H$}{H}}
\label{sec:lem_GH}

\subsection{Preparing the Graph \texorpdfstring{$G$}{G} (Proof of Lemma \ref{lem:G_lemma})}
\label{sec:robust_G_lemma}  

In this section, we provide a proof of Lemma~\ref{lem:G_lemma}. We follow a standard proof strategy, used for instance in~\cite{dense_bandwidth_theorem}. Given an $n$-vertex graph~$G$ with minimum degree $\delta(G) \geq (\exth + \gamma)n$, we can use the degree version of the regularity lemma to obtain a reduced graph~$R$ with a high minimum degree. By the clique version of the Hajnal--Szemer\'{e}di theorem (Corollary~\ref{cor:clique_hajnal_szemeredi}), the graph~$R$ contains a $K_{\Delta+1}$-factor that covers all but at most~$\Delta$ vertices of~$R$. This $K_{\Delta+1}$-factor emerges as a natural candidate for~$R'$. We shall see in the next section that this choice of~$R'$ will aid in the distribution of the chosen buffer vertices in a manner that satisfies~\ref{itm:SpBL3}. The bounded maximum degree of~$K_{\Delta + 1}$ is also useful to establish super-regularity on~$R'$ using Lemma~\ref{lem:super-regularity}. Finally, the graph~$R'$ can be made spanning for the graph~$R$ by slightly changing~$R$ by removing the parts not in~$R'$ and redistributing the vertices of~$G$ in these parts to other parts of~$R$. Lemma~\ref{lem:robust_regularity} ensures that this redistribution does not affect the regularity and super-regularity properties constructed so far. 

We now give a proof of Lemma \ref{lem:G_lemma}.

\begin{proof}[Proof of Lemma \ref{lem:G_lemma}]
 Given~$\gamma > 0$ and~$\Delta \geq 2$, Let~$G$ be an $n$-vertex graph with minimum degree $\delta(G) \geq (\exth + \gamma)n$. For this~$\gamma$, an application of Corollary~\ref{cor:reduced-min-degree} returns constants~$d^*_0$ and~$\eps^*_0$ for which the statement of the corollary holds. Pick the constants $d_0 = \min \bigl\{ d_0^*/2 , \gamma/16 \bigr\}$ and~$\kappa = 2$. Now, let $\eps > 0$ and $0 < d < d_0$ be given as in the preamble of Lemma~\ref{lem:G_lemma}, and for these choices of~$\eps$ and~$d$, pick two auxiliary constants~$\eps'$ and~$d'$ as follows.  
 $$\eps'  \coloneqq   \frac{\eps^4\gamma^2 (d_0^* - 2d)^2}{3 \cdot (8\Delta)^4}\, \text{ and }
   d' \coloneqq 2d + 4\eps' \Delta + 4\Delta\sqrt{\eps'}.$$
\indent Note that $d' \leq \gamma /4$. Moreover, $d' \leq 2d + 8 \sqrt{\eps'}\Delta \leq (2d + d_0^*)/2 < d_0^*$, and thus, applying the degree version of the regularity lemma (Theorem~\ref{thm:regularity} and Corollary~\ref{cor:reduced-min-degree}), with constants $\eps'$ and $d'$, and some integer $m_0 > \max\{1/\eps', 4\Delta/\gamma\}$, we get an integer~$M_0$, such that there exist a spanning subgraph~$G'$ of~$G$ and a partition~$\mathcal{U} = U_0 \sqcup U_1 \sqcup \dots \sqcup U_m$ of~$V(G)$ for $m_0 \leq m \leq M_0$ for which \ref{itm:RL1}--\ref{itm:RL6} hold. Let~$r_0 \coloneqq  M_0$ be the required integer to be chosen for Lemma~\ref{lem:G_lemma}. By the choice of~$m_0$ and \ref{itm:RL4}, it follows that $|U_i| \leq \min\{n/m, \eps n\}$ for each part~$U_i$ of~$\mathcal{U}$.  

 Let~$R_0$ be the reduced graph generated by the above partition~$\mathcal{U}$. By Corollary~\ref{cor:reduced-min-degree}, the graph~$R_0$ on~$m$ vertices has a minimum degree $\delta(R_0) \geq \big(\exth + \gamma/2\big)m$. Then as $\exth\ge \Delta/(\Delta+1)$, by Corollary~\ref{cor:clique_hajnal_szemeredi}, the graph~$R_0$ contains a~$K_{\Delta+1}$-factor, covering all but at most~$\Delta$ vertices of~$R_0$. Let~$R'$ be this $K_{\Delta+1}$-factor in~$R_0$ and let~$R$ be the induced subgraph of~$R_0$ on the vertex set~$V(R')$. These will be the reduced graphs~$R$ and~$R'$ required by Lemma~\ref{lem:G_lemma}. 
 
 Now, given~$R$ and~$R'$, we need to construct a partition $\mathcal{V} = \{V_i\}_{i\in [|R|]}$ from the partition $\mathcal{U}$, which satisfies the properties \ref{itm:G1}--\ref{itm:G3}. To do this we begin by relabelling the vertices of~$R'$ by assigning to each vertex, a pair~$(i,j)$ that represents the clique component it lies in and its vertex number within this component. This pair~$(i,j)$ lies in the set~$[\Delta +1]\times[c]$, where~$c$ equals the number of $K_{\Delta+1}$-components in~$R'$. Note that we have $m \geq c(\Delta +1)$ and that $m - c(\Delta +1) \leq \Delta$. 
 
 We will first alter the partition~$\mathcal{U}$ to attain super-regularity on $E(R')$. An application of Lemma~\ref{lem:super-regularity} to the reduced graph~$R_0$ with subgraph~$R'$ of bounded degree $\Delta(R') \leq \Delta$, shows that for every $(i,j) \in [\Delta+1] \times [c]$, there exists a subset~$U'_{(i,j)} \subseteq U_{(i,j)}$ of size $|U'_{(i,j)}| = (1-\eps'\Delta)|U_{(i,j)}|$ such that the collection of disjoint subsets $\bigl\{U'_{(i,j)}: (i,j) \in [\Delta+1]\times[c]\bigr\}$ is $(\eps'/(1-\eps'\Delta), d'-2\eps'\Delta)$-super-regular on~$R'$ and $(\eps'/(1-\eps'\Delta), d'-2\eps'\Delta)$-regular on~$R_0$ (and hence on $R$ as well). Here, we have used the fact that~$\Delta+1 \leq 2\Delta$. 
 
 Next, for each node $(i,j) \in V(R_0)$ we shall move the vertices of~$U_{(i,j)} \setminus U'_{(i,j)}$ to the exceptional set $U_0$. Similarly, for each node $k \in V(R_0)$ not covered by $R'$, move all vertices of $U_k$ to the set $U_0$. This gives a partition $\mathcal{U}' = U'_0 \sqcup U'_{(1,1)} \sqcup \dots \sqcup U'_{(\Delta + 1, c)}$, which is $(\eps'/(1-\eps'\Delta), d'-2\eps'\Delta)$-super-regular on~$R'$ and $(\eps'/(1-\eps'\Delta), d'-2\eps'\Delta)$-regular on~$R$. Note that by~\ref{itm:RL4}, the sets~$U'_{(i,j)}$ have the same cardinality $L' = |U'_{(i,j)}| = (1-\eps'\Delta)|U_{(i,j)}|$, for all $(i,j) \in [\Delta+1]\times[c]$. 
 
 Finally, we need to deal with the exceptional set to get the required partition~$\mathcal{V}$. Note that the enlarged exceptional set~$U'_0$ is seen to be small as follows.
    $$|U'_0| \leq |U_0| + \;\; \sum_{\mathclap{(i,j) \in V(R')}}\; |U_{(i,j)} \setminus U'_{(i,j)}| + \;\;\sum_{\mathclap{k \in V(R_0\setminus R)}}\; |U_k| \leq \eps'n + m \left(\eps' \Delta \frac{n}{m}\right) + \Delta\, \eps' n \leq 3\eps' \Delta n.$$
\indent We will now redistribute these vertices. We ultimately want to use Lemma~\ref{lem:robust_regularity} to show that regularity and super-regularity are preserved after this redistribution of the vertices of the exceptional set~$U_0$. For this, we also need to ensure that the minimum degree conditions across the pairs~$\bigl( U'_{(i,j)}, U'_{(i',j)}\bigr)$ for all~$i, i' \in [\Delta+1]$ and $j \in [c]$ representing the edges of~$R'$, are preserved after this redistribution. Hence, for each vertex~$v \in U'_0$, we say that a clique component~$j \in [c]$ of~$R'$ is \emph{$v$-friendly} if~$v$ has at least~$d'n/r$ neighbours in each part~$U'_{(i,j)}$ for $i \in [\Delta +1]$, where $r = |V(R)| = c(\Delta+1)$. Now, if a vertex~$v \in U'_0$ has at most~$\gamma c$ clique components that are $v$-friendly, then the degree of~$v$ in~$G'$ is at most 
\begin{eqnarray*}
    d_{G'}(v) &\leq&  \gamma c (\Delta + 1) L'  + (1-\gamma)c \biggl(\Delta L' + \frac{d'n}{r} \biggr) + 3\eps'\Delta n\\
    &=& (\Delta + \gamma) cL' + (1-\gamma)d' \frac{n}{\Delta+1} + 3\eps' \Delta n \quad \leq \bigg(\frac{\Delta+\gamma}{\Delta+1} + \frac{(1-\gamma)d'}{\Delta+1} + 3\eps'\Delta\bigg)n \\
    &\leq& \bigg(\frac{\Delta}{\Delta + 1} + \frac{\gamma}{\Delta+1} + \frac{\gamma}{4(\Delta+1)} + \frac{\gamma}{8(\Delta+1)} \bigg)n \leq \big( \exth + 0.46 \gamma \big)n\, ,
\end{eqnarray*}
where the three terms in the first inequality represent the maximum number of neighbours of~$v$ in parts of friendly clique components, in parts of non-friendly clique components, and in the exceptional set~$U'_0$, respectively. This contradicts~\ref{itm:RL2} as since $d' + \eps' \leq \gamma/2$, we know that~$G'$ has minimum degree at least $\delta(G') \geq \delta(G) - (d'+\eps')n \geq \bigl( \exth + \gamma/2\bigr)n$. Hence, every vertex~$v \in U'_0$ has at least~$\gamma c$ clique components in~$R'$ that are $v$-friendly. 

Begin by assigning each vertex~$v \in U'_0$ to one of their $v$-friendly clique components as uniformly as possible, such that each component is assigned at most $|U'_0|/\gamma c$ vertices. Then, for each clique component, distribute its assigned vertices equitably among its~$\Delta+1$ parts. As a result, the number of vertices of~$U_0'$ received by each part~$U'_{(i,j)}$ is at most 
\begin{eqnarray*}
    \frac{|U'_0|}{\gamma c(\Delta+1)} &\leq& \frac{3 \eps'\Delta}{\gamma c (\Delta + 1)}n \ \leq  \frac{3\eps' \Delta}{\gamma c (\Delta+1)} \cdot \frac{m}{(1-\eps'\Delta)(1-\eps')} L'\\
    &\leq& \frac{3\eps' \Delta}{\gamma}\cdot \frac{2}{1- 2\eps'\Delta}L'\leq \frac{12\Delta}{\gamma} \eps' L' \leq \sqrt{\eps'}L'\, .
\end{eqnarray*}

Label these newly formed parts by $V_{(i,j)}$. We will now show that $\mathcal{V} = \{V_{(i,j)}\}$ is the required partition. By Lemma~\ref{lem:robust_regularity}, with $\hat{\alpha} = \hat{\beta} = \sqrt{\eps'}$, we see that the pairs $(V_{(i,j)}, V_{(i',j')})$ are $(\eps, d)$-regular on $R$, due to the following relations. 
$${\eps'}/({1-\eps'\Delta}) + 6(\eps')^{1/4} \leq 2\eps' + 6(\eps')^{1/4} \leq 8(\eps')^{1/4} \leq \eps \text{ and, }$$
$$ \big(d' - 2\eps'\Delta\big) - 4\sqrt{\eps'} \geq 2d + 2\eps'\Delta \geq d.$$

Now, for super-regularity, note that as the partition~$\mathcal{U}'$ is $(\eps'/(1-\eps'\Delta),d'-2\eps'\Delta)$-regular on~$R$ and $(\eps'/(1-\eps'\Delta), d'-2\eps'\Delta)$-super-regular on~$R'$, it is also $(\eps'/(1-\eps'\Delta), 0.5 d'-\eps'\Delta)$-regular on~$R$ and $(\eps'/(1-\eps'\Delta), 0.5d'-\eps'\Delta)$-super-regular on~$R'$. Moreover, each vertex in~$U'_{(i,j)}$ has at least~$(d' - 2\eps'\Delta)L'$ neighbours in~$U'_{(i',j)}$ for all $i \neq i' \in [\Delta+1]$ and $j \in [c]$, and hence, as $|V_{(i',j)}| \le (1+\sqrt{\eps'})L' \leq 2L'$, each such vertex of $U'_{(i,j)}$ also has at least $(0.5d' - \eps'\Delta)|V_{(i',j)}|$ neighbours in~$V_{(i',j)}$. On the other hand, as all vertices of the exceptional set have been added to parts in friendly clique components and as~$d' \geq \sqrt{\eps'}\Delta(1 + \sqrt{\eps'})$, every vertex in $V_{(i,j)}\setminus U_{(i,j)}$ also has at least 
\begin{eqnarray*}
    d'n/r &\geq& d'L' \geq d'L' + \sqrt{\eps'}\big(d' - \sqrt{\eps'}\Delta - \eps'\Delta\big)L'\\
    &=&(d'-\eps'\Delta)(1 + \sqrt{\eps'})L' \geq (0.5d' - \eps'\Delta)|V_{(i',j)}|
\end{eqnarray*} 
neighbours in~$V_{(i',j)}$ for~$i\neq i'$. Since $\Delta \geq 2$, we have $(0.5d' - \eps'\Delta) - 4\sqrt{\eps'} \geq d$, and so by Lemma \ref{lem:robust_regularity}, the partition $\mathcal{V}$ is also $(\eps, d)$-super-regular on~$R'$. Thus the graph~$G'$ satisfies~\ref{itm:G3}. The condition~\ref{itm:G2} follows from the following observation.
$$\delta(R) \geq \delta(R_0) - \Delta \geq \delta(R_0) - (\gamma/4)m \geq \bigl( \exth + \gamma/4\bigr)m \geq \bigl(\exth + \gamma/4\bigr)r\,.$$ 

Finally, for~\ref{itm:G1}, it is easy to see that~$\mathcal{V}$ is $\kappa$-balanced for $\kappa = 2$ as 
\[L' \leq |U'_{(i,j)}| \leq |V_{(i,j)}| \leq (1 + \sqrt{\eps'})L' \leq 2L', \quad \text{ for all } (i,j) \in [\Delta + 1]\times[c].\]

Further, as $|U'_0| \leq 3\eps'\Delta n$, we have that for all $(i,j) \in [\Delta+1]\times [c]$, the sets~$V_{(i,j)}$ have size at least $|V_{(i,j)}| \geq L' \geq (1 - 3\eps'\Delta)n/r \geq (1 - \gamma/8)n/r \geq n/\kappa r$, and at most $|V_{(i,j)}| \leq 2L' \leq \kappa n/r$ for $\kappa = 2$ and $\gamma < 1- \exth \leq 1/2$, thereby completing the proof. 
\end{proof}

\subsection{Partitioning the graph \texorpdfstring{$H$}{H} (Proof of Lemma \ref{lem:Lemma_H})}
\label{sec:robust_H_lemma}

Let us begin with a brief sketch of the proof of Lemma~\ref{lem:Lemma_H}. Modulo the presence of the given pre-partition~$X^*$, the conclusion of Lemma~\ref{lem:Lemma_H} can be achieved in three steps. First, we begin by identifying a good subset of vertices $\mathcal{B} \subseteq V(H)$, which will act as potential buffer vertices. For this, let $N^2[x] \coloneqq \{y \in V(H): \dist(x,y) \leq 2\}$ denote the \emph{closed second neighbourhood} of the vertex $x \in V(H)$. Note that in order to satisfy the conclusion \ref{itm:H3}, we require that for any buffer vertex $x \in \tilde{X}_i$, if $xy, yz \in E(H)$ for some $y\in X_j$ and $z \in X_k$, then $ij$ and $jk$ must be edges of the reduced graph $R'$. As $R'$ is a disjoint union of cliques $K_{\ell}$ for $\ell \geq \Delta+1$, we need to ensure that for each buffer vertex $x \in \tilde{X}_i$, all vertices of $N^2[x]$ must be assigned to parts of the same clique component of $R'$ which contains the node $i$, while ensuring that the assignment forms an $R$-partition. 

For such an assignment, it will be convenient to pick buffer vertices such that for every pair $x, y \in \mathcal{B}$ there are no edges between the sets $N^2[x]$ and $N^2[y]$, as this will allow for $N^2[x]$ to be partitioned and assigned independently of $N^2[y]$. This can be achieved by constructing an auxiliary graph~$\tilde{H}$, where $xy \in E(\tilde{H})$ if and only if $\dist_H(x,y) \leq 5$. Now, any independent set in the graph $\tilde{H}$ can be chosen as the set of potential buffer vertices $\mathcal{B}$, which will then be equitably distributed to form the sets $\tilde{X}_i$. A large independent set of the required size in $\tilde{H}$ will be found by an application of the Hajnal--Szemer\'{e}di theorem (Theorem \ref{thm:Hajnal-Szemeredi}) to the graph $\tilde{H}$. 

Once the potential buffer sets $\tilde{X}_i$ have been fixed, the second step will be to assign the vertices of $N^2[x]$ for each $x \in \tilde{X}_i$ to parts of the clique $K_{\ell}$ of $R'$ which contains the node $i$. To do this, we apply Theorem~\ref{thm:Hajnal-Szemeredi} to the subgraph of $H$ induced by $N^2[x]$ and assign the obtained independent sets to distinct parts of the clique. 

Finally, we need to construct the partition $\mathcal{X} = \{X_i\}_{i\in [r]}$ by assigning the remaining vertices of~ $V(H)$ to some part $X_i$, so that $\mathcal{X}$ forms an $R$-partition and is size-compatible to the partition $\mathcal{V}$ of $V(G)$. This assignment can be completed by the definition of $\exth$ as follows: Let $R^*$ be the blow-up of the reduced graph $R$ with respect to the function $f: i \mapsto |V_i|$. The previous two steps give a partial embedding of the induced subgraph $H\bigl[N^2[\mathcal{B}]\bigr]$ of $H$ into $R^*$. Moreover, $R^*$ has a high minimum degree by the lower bound assumptions on $\delta(R)$ and $|V_i|$ in the statement of Lemma~\ref{lem:Lemma_H}. Then, by the definition of $\exth$, this partial embedding extends to an embedding of $H$ into $R^*$, thereby giving a size-compatible $R$-partition of $V(H)$. It is now easy to check that both \ref{itm:H2} and \ref{itm:H3} are satisfied for the partitions $\mathcal{X}$ and $\tilde{\mathcal{X}}$. The small fixed pre-partition $\mathcal{X}^*$ will not majorly affect this outlined strategy and can be dealt with comfortably. 

We are now ready to prove Lemma \ref{lem:Lemma_H}.

\begin{proof}[Proof of Lemma \ref{lem:Lemma_H}]
    Given $\Delta \geq 2$, $\gamma > 0$, and $\kappa > 1$,
    we let $\delta>0$ be returned by the definition of $\exth$ with input $\gamma/2$. That is, for any $n$-vertex graphs $G'$ and $H'$ with $\Delta(H')\le\Delta$ and $\delta(G')\ge(\exth+\gamma/2)\,n$, any subset $S'\subseteq V(H')$ with $|S'|\le\delta n$, and any embedding $\varphi_{S'}$ of $H'[S']$ into $G'$, there is an embedding $\varphi$ of $H'$ into $G'$ that extends $\varphi_{S'}$.
    
    Let $\alpha \coloneq \gamma^2\delta/8 \kappa \Delta^6$ be the constant required by Lemma \ref{lem:Lemma_H}. For the given $\ell\geq \Delta+1$, suppose that the graphs $G$, $R$, $R'$, and partitions $\mathcal{V} = \{V_i\}_{i \in [r]}$ of $V(G)$ and $\mathcal{X}^* = \{X_i^*\}$ of $X^* \subseteq V(H)$ are given as in the statement of Lemma \ref{lem:Lemma_H}. Let the function $f: V(R) \to \mathbb{N}$ be defined by the mapping $i \mapsto |V_i|$ for all $i \in [r]$, and let $R^*$ denote the blow-up graph of $R$ with respect to the function $f$. For each $i \in [r]$, define $F_i$ to be the set of $|V_i|$ independent vertices representing the blow-up of the node $i \in V(R)$ in $R^*$. We construct the partition $\mathcal{X}$ and the sets $\tilde{X}_i$ as follows. Let $U \subseteq V(H)$ be the set of vertices $\{ u \in V(H) : \dist_H(u, x^*) > 3  \text{ for all } x^* \in X^*\}$. As for any vertex $x^* \in X^*$, there are at most $1 + \Delta + \Delta^2 + \Delta^3 \leq \Delta^4$ vertices $v$ with distance $\dist_H(v,x^*) \leq 3$, we have that $|U| \geq n - \Delta^4|X^*| \geq (1 - \gamma/4)n$.

    \textbf{Step I: Constructing the sets $\tilde{X}_i$.} Given the graph $H$, consider an auxiliary graph $\tilde{H}$ with $V(\tilde{H}) = U \subseteq V(H)$ and $xy \in E(\tilde{H})$ if an only if $\dist_H(x,y) \leq 5$. Since there are at most $\tilde{\Delta} = 1 + \Delta + \Delta^2 + \dots + \Delta^5$ vertices at a distance at most $5$ from any fixed vertex of $H$, it follows that the maximum degree $\Delta(\tilde{H}) \leq \tilde{\Delta} \leq \Delta^6 - 1$. Hence, by Theorem~\ref{thm:Hajnal-Szemeredi}, the vertices $U = V(\tilde{H})$ can be equitably partitioned into $\Delta^6$ independent sets of $\tilde{H}$. Let $\mathcal{B}_0$ be one of these independent sets. We will choose some vertices from $\mathcal{B}_0$ to form the potential buffer vertices for $H$. The sets $\tilde{X}_i$ are formed by choosing distinct $\bigl\lceil \alpha |V_i| \bigr\rceil$ vertices from the set $\mathcal{B}_0$ for each $i \in [r]$. This is possible as 
    $$\sum_{i \in [r]} \bigl\lceil \alpha |V_i| \bigr\rceil \leq 2\alpha n \leq  (1 - \gamma/4) n/\Delta^6  \leq |\mathcal{B}_0|.$$ 
    Define $\mathcal{B} = \bigcup_{i \in [r]} \tilde{X}_i$ to be the set of chosen potential buffer vertices. Let $\varphi_{\mathcal{B}}: \mathcal{B} \to V(R^*)$ denote an arbitrary map that for every $i \in [r]$ sends each $x \in \tilde{X}_i$ to some distinct vertex in $F_i$. Note that by choice of $U$, as $\dist_H(u, x^*) > 3$ for any $u \in U$ and $x^* \in X^*$, we have that $N^2[\mathcal{B}] \cap X^* = \emptyset$ and that there is no edge between vertices of $X^*$ and $N^2[\mathcal{B}]$. 
    
    \textbf{Step II: Partially embedding $N^2[\mathcal{B}]$ into $R^*$.} Without loss of generality, consider a buffer vertex $x \in \tilde{X}_1$, and let the nodes $\{1, 2, \dots, \ell\} \subseteq [r]$ denote the clique component $K_{\ell}$ in $R'$ for the given $\ell \geq \Delta+1$, which contains the node $1 \in V(R)$. We wish to assign each vertex of $N^2[x] \setminus\{x\}$ to some distinct available vertex of $F_i$ for $i \in [\ell]$. For this, consider the induced subgraph $H_x = H\bigl[N^2[x]\bigr]$, which induces the maximum degree condition $\Delta(H_x) \leq \Delta \leq \ell-1$ from $H$. By Theorem~\ref{thm:Hajnal-Szemeredi}, $N^2[x] = V(H_x)$ can be equitably partitioned into $\ell$ independent sets of $H_x$. Map the vertices of the independent set containing $x$ to distinct available vertices of $F_1 \subseteq V(R^*)$. Arbitrarily assign the remaining $\ell-1$ independent sets to each of the nodes $\{2, 3, \dots, \ell\}$, and for each node $i \in \{2, 3, \dots, \ell\}$, map the vertices of the assigned independent set to distinct available vertices of $F_i \in V(R^*)$. By construction of the set $\mathcal{B}$, this process can be independently repeated for each buffer vertex $x \in \mathcal{B}$, as for $x \neq y \in \mathcal{B}$ we have that $\dist_H(x,y) > 5$ and hence, the sets $N^2[x]$ and $N^2[y]$ are disjoint with no edges going between them. 
    
    Note that it is possible to map the whole of $N^2[\mathcal{B}]$ to distinct vertices of $R^*$. Indeed, each clique component receives at most $\ell\cdot 2\alpha \kappa n/r$ buffer vertices, each of these buffer vertices has a set $N^2[v]$ of size at most $1 + \Delta + \Delta^2 \leq \Delta^3$, and all these vertices are equitably divided among the $\ell$ nodes of the clique component. Thus, a total of at most $2\Delta^3\alpha \kappa n/r \leq \gamma|V_i|/2\Delta^3 < |F_i|/4$ vertices are assigned to each of the sets $F_i$. Similarly, we assign the vertices of $X_i^*$ to distinct available vertices of $F_i$, which is also possible as there is no edge between $N^2[\mathcal{B}]$ and $X^*$, the set $X_i^*$ has size at most $|F_i|/4$, and only at most $|F_i|/4$ vertices have been assigned to $F_i$ so far. 
    
    Now, let $\varphi_S : N^{2}[\mathcal{B}] \cup X^* \to V(R^*)$ denote this mapping of vertices of $N^2[\mathcal{B}]$ and $X^*$. By construction, all vertices mapped to $F_i$ form an independent set in $H$ for each $i \in [r]$, and as the graphs $R^*[F_i \sqcup F_j]$ are complete bipartite graphs for all $ij \in E(R)$, the map $\varphi_S$ embeds the graph $H[N^2[\mathcal{B}]\cup X^*]$ into $R^*$.  

    \textbf{Step III: Constructing the partition $\mathcal{X}$.} For the set $S = N^2[\mathcal{B}] \cup X^* \subseteq V(H)$, we have an embedding $\varphi_S$ of $H[S]$ into $R^*$. Further, the set $S$ has size at most $|S| = |N^2[\mathcal{B}]| + |X^*| \leq 2\alpha n \cdot \Delta^3 + \alpha n \leq \delta n $. As  $\delta(R) \geq (\exth + \gamma)r$, and $|V_i| \geq \bigl(1 - \frac{\gamma}{2}\bigr)\frac{n}{r}$, we have that for all $u \in V(R^*)$, 
    $$ d_{R^*}(u) \geq \delta(R) \cdot |V_i| \geq (\exth + \gamma)r \cdot \Bigl(1 - \frac{\gamma}{2}\Bigr)\frac{n}{r} \geq \Bigl(\exth + \gamma - \frac{\gamma}{2}\Bigr) n = \Bigl(\exth + \frac{\gamma}{2}\Bigr) n.$$
    Thus by definition of $\exth$, there is an embedding $\varphi \in \Emb(H, R^*)$ which extends the partial map $\varphi_S$. We then define the partition $\mathcal{X} = \{X_i\}_{i \in [r]}$ by setting $X_i = \varphi^{-1}(F_i)$ for all $i \in [r]$. 
    
    It now remains to verify \ref{itm:H1}--\ref{itm:H4}: The partition $\mathcal{X}$ is size-compatible to $\mathcal{V}$ as $\varphi$ is a bijection, and hence, for every $i \in [r]$, we have $|X_i| = |F_i| = |V_i|$. Further, suppose $x \in X_i$ and $y \in X_j$ are such that $xy \in E(H)$. Then, as $(H[X^*], \mathcal{X}^*)$ is an $R$-partition, and by the extension of $\varphi_S$ above, we have $\varphi(x)\varphi(y) \in E(F_i, F_j)$ and hence $ij \in E(R)$, making $(H, \mathcal{X})$ an $R$-partition. This verifies \ref{itm:H1} and \ref{itm:H2}. For \ref{itm:H3}, by Step I, each $X_i$ contains at least $|\tilde{X}_i| \geq \alpha|V_i| = \alpha|X_i|$ buffer vertices, and by the placement of $N^2[\mathcal{B}]$ in the clique components of $R'$ in Step II, if for $x \in \tilde{X}_i$ we have $y \in X_j, z \in X_k$ with $xy, yz \in E(H)$, then $ij, jk \in E(R')$. Thus, $\tilde{\mathcal{X}}$ is an $(\alpha, R')$-buffer for $(H, \mathcal{X})$. Finally, \ref{itm:H4} follows from the choice of set $U$ used to choose the potential buffer vertices $\mathcal{B}$.\\ 
\end{proof}

\section{Concluding Remarks and Optimality}
\label{sec:conclusion}

It might be possible to improve Theorem~\ref{thm:Robust_main_basic} in two
directions. First, we believe that if Conjecture~\ref{conj:BEC}
holds true, then it should do so robustly. Thus, it seems reasonable
to conjecture the following.

\begin{conj}
  The extension threshold~$\exth$ equals $\Delta/(\Delta+1)$.
\end{conj}

If this conjecture is true, then our
bound on the minimum degree of~$G$ in Theorem~\ref{thm:Robust_main_basic} is optimal up to a multiplicative factor of at most~$2$.

The second direction concerns our range of values for~$p$. A simple lower
bound on the values of~$p$ which satisfy Theorem~\ref{thm:Robust_main}
is given by the threshold for containment of the graph~$H$ in the Erd\H
os--R\'{e}nyi random graph~$G(n, p) \coloneqq K_n(p)$. For indeed, if~$H$ is not contained in~$G(n,p)$ with high probability, then it is unlikely to be contained in the graph~$G(p)$ as well. This lower
bound can act as a measure that is indicative of the degree of robustness of some given
property~$\Pi$. More specifically,~$\Pi$ is said to be \emph{extremely robust} for~$G$ if the
threshold for~$G(p)$ to satisfy~$\Pi$ matches the
threshold for~$G(n,p)$ to satisfy~$\Pi$ up to a multiplicative constant (see
\cite{toolkit_robust_thresholds}). For instance,
the robustness result in~\cite{robust_hamiltonicity} says that Hamiltonicity is extremely robust for Dirac graphs, as the threshold for Hamiltonicity 
for~$G(n,p)$ equals~$\Theta(\log n/n)$. 

In fact, it is believed that the property of containing a graph~$H$ is extremely robust for any host graph $G$ that satisfies the minimum degree threshold required for containing~$H$. This conjecture, which was raised in \cite{Alp_robustness_spread, toolkit_robust_thresholds}, puts forth the idea that beyond a certain minimum degree condition for~$G$, any obstruction to embedding a given fixed subgraph in~$G(p)$ is entirely independent of the host graph~$G$ and relies solely on the properties of~$G(n,p)$.  

In the case of Theorem~\ref{thm:Robust_main}, the exact threshold for
containing a subgraph~$H$ is not entirely known. However, for any possible value of $m_1$, if there exists a graph $H'$ with $m_1(H') = m_1$, then there is a sequence of graphs given by $H'$-factors for which the threshold for containment in $G(n,p)$ is of the order $\Omega(n^{-1/m_1})$ \cite{K_rfactor_threshold}. 
Thus, over the class of bounded degree graphs~$H$ having $\Delta(H) \leq \Delta$ and with a fixed maximum $1$-density $m_1(H)$, our value for~$p$, and hence
the strength of our robustness result, is optimal up to the $\log$
factor.

Further, it is not difficult to check that among all graphs~$H$ with maximum
degree $\Delta\ge2$, the quantity~$m_1(H)$ is maximised by
$K_{\Delta+1}$ (albeit not uniquely), which has
$m_1(H)=(\Delta+1)/{2}$. For all graphs~$H$ with
$\Delta(H) \leq \Delta$ and $m_1(H)={(\Delta+1)}/{2}$, we are able to work with the slightly better
probability $Cn^{-2/(\Delta+1)}(\log n)^{1/\binom{\Delta+1}{2}}$.

\begin{theorem}
    \label{thm:tight_improvement}
    For all\/ $\gamma > 0$ and\/ $\Delta \in \mathbb{N}$, there exists
    a constant\/ $C^* > 0$ such that if~$H$ is an $n$-vertex graph
    with maximum degree\/ $\Delta(H) \leq \Delta$ and with $m_1(H) = (\Delta+1)/2$, and~$G$ is an\/
    $n$-vertex graph with minimum degree\/
    $\delta(G) \geq (\exth+ \gamma) n$, then for\/
    $p \geq C^*n^{-2/(\Delta+1)}(\log n)^{1/\binom{\Delta+1}{2}}$, with high probability, the graph~$H$ is contained in the random subgraph~$G(p)$ of~$G$.
\end{theorem}

Johansson, Kahn, and Vu~\cite{K_rfactor_threshold} showed that the (sharp) threshold for containing a $K_{\Delta+1}$-factor in $G(n,p)$ is of the order of magnitude $n^{-2/(\Delta+1)}(\log n)^{1/\binom{\Delta+1}{2}}$. This makes our value for $p$ in Theorem~\ref{thm:tight_improvement} optimal up to the
constant over the class of graphs $H$ with $\Delta(H) \leq \Delta$ and $m_1(H) = (\Delta+1)/2$. The remainder of this section is devoted to a proof of Theorem~\ref{thm:tight_improvement}. We begin with the following observation.

\begin{lemma}
\label{lem:suplementary}
    Let $H$ be an $n$-vertex graph with $\Delta(H) \leq \Delta$. If~$H$ contains no copy of~$K_{\Delta+1}$ as a subgraph, then we have that $n^{-1/m_1(H)}\log n = o\bigl(n^{-2/(\Delta+1)}\bigr)$. 
\end{lemma}
\begin{proof}
   Note that for any graph~$H$, its maximum $1$-density~$m_1(H)$ is at most $(\Delta(H) +1)/2$. Now let $H$ be an $n$-vertex $K_{\Delta+1}$-free graph with $\Delta(H) \leq \Delta$ as given. It suffices to show that $m_1(H) \leq (\Delta+1)/2 - \eps_{\Delta}$, for  $\eps_{\Delta} = 1/2(\Delta+1)$. Indeed, if this were true, then there exists a constant~$\eps_{\Delta}^*$ such that $n^{-1/m_1(H)} = \bigO\bigl(n^{-(2/(\Delta+1) + \eps_{\Delta}^*)}\,\bigr)$, and so the statement of the lemma follows.  

    To prove the claimed bound on~$m_1(H)$, let~$H'$ be any subgraph of~$H$ and let $S \subseteq V(H)$ be its vertex set. 
    It suffices to show that the $1$-density $d_1(H')$ is less than the afore-claimed bound. 
    If $|S| \leq \Delta$, then $\Delta(H') \leq \Delta-1$, and hence we have that $d_1(H') \leq m_1(H') \leq \Delta/2$, as required. 
    If $|S| = \Delta+1$, then $H'$ can have at most $\binom{\Delta+1}{2} -1$ edges as $H$ is $K_{\Delta+1}$-free, and so we have that $d_1(H') = |E(H')|/\Delta \leq (\Delta+1)/2 - 1/\Delta$. 
    Finally, assume that $|S| > \Delta+1$. Then, as $\Delta(H) \leq \Delta$, the graph~$H'$ has at most $\Delta|S|/2$ edges, and so the claimed bound holds since 
    \[d_1(H') \leq \frac{\Delta|S|}{2(|S|-1)} \leq \frac{\Delta(\Delta+2)}{2(\Delta+1)} = \frac{(\Delta+1)^2 - 1}{2(\Delta+1)} = \frac{(\Delta+1)}2 - \eps_{\Delta}. \qedhere\] 
    \end{proof}

In order to prove Theorem~\ref{thm:tight_improvement}, we begin by partitioning the graph~$H$ into two vertex-disjoint subgraphs~$H_1$ and~$H_2$, where~$H_1$ is defined to be the subgraph formed by all copies of~$K_{\Delta+1}$ in~$H$ and~$H_2$ to be the remaining $K_{\Delta+1}$-free subgraph of~$H$. Note that as $\Delta(H) \leq \Delta$, the graph~$H_1$ is a disjoint union of cliques~$K_{\Delta+1}$. These cliques form connected components in~$H$, and thus are also vertex disjoint from the graph~$H_2$. 

Lemma~\ref{lem:suplementary} shows that any value of $p$ satisfying the conditions of Theorem~\ref{thm:tight_improvement} can be used in conjunction with Theorem~\ref{thm:Robust_main} to find a copy of~$H_2$ in~$G$. On the other hand, we use the following result by Pham, Sah, Sawhney, and Simkin to find a copy of~$H_1$ in~$G$ \mcite[Theorem 1.7 (2)]{toolkit_robust_thresholds}. 

\begin{theorem}
\label{thm:toolkit_robust}
For every $\Delta \in \mathbb{N}$, there exists a constant~$C' > 0$ such that if~$G$ is a graph on $n \in (\Delta+1)\mathbb{N}$ vertices with $\delta(G) \geq \frac{\Delta}{\Delta+1}n$, then for $p \geq C'\, n^{-2/(\Delta+1)} (\log n)^{1/\binom{\Delta+1}{2}}$, with high probability, the random subgraph $G(p)$ contains a $K_{\Delta+1}$-factor.
\end{theorem}

Finally, we use a Chernoff bound for hypergeometrically distributed random variables to find vertex-disjoint copies of~$H_1$ and~$H_2$ in~$G$, thereby proving Theorem~\ref{thm:tight_improvement}. Recall that a random variable~$\mathbb{X}$ is said to follow a \emph{hypergeometric distribution} with parameters $(n, m, k)$ if given a set~$\mathcal{U}$ of size~$n$ and a fixed subset~$\mathcal{A} \subseteq \mathcal{U}$ of size~$m$, the random variable~$\mathbb{X}$ counts the size of the intersection~$\mathcal{A} \cap \mathcal{T}$ for a subset $\mathcal{T} \subseteq \mathcal{U}$ of size~$k$ generated uniformly at random among all the $\binom{n}{k}$~such subsets of~$\mathcal{U}$. Note that the expected value~${\rm E}(\mathbb{X})$ of~$\mathbb{X}$ equals~$mk/n$. 
The following concentration bound holds for such a random variable~$\mathbb{X}$~\cite[Chapter 2]{Random_graphs_textbook_janson}.

\begin{lemma}
[Hypergeometric Chernoff Bound]
\label{lem:chern_hypgeom}
Let $\mathbb{X}$ be a hypergeometrically distributed random variable with parameters $(n, m, k)$. Then for any $\eps \in (0,1)$ and for $\eps\,{\rm E}(\mathbb{X})\leq t \leq {\rm E}(\mathbb{X})$, we have
$$\mathds{P}\Bigl(\bigl|\mathbb{X} - {\rm E}(\mathbb{X})\bigr| \geq t \Bigr) \leq 2e^{-\eps^2t/3}\,.$$
\end{lemma}

We are now ready to prove Theorem~\ref{thm:tight_improvement}.

\begin{proof}[Proof of Theorem~\ref{thm:tight_improvement}]
    
    Let the constants $\Delta$ and $\gamma$ be given. With input $\Delta$ and $\gamma/2$, Theorem~\ref{thm:Robust_main} returns a constant~$C$, and with input $\Delta$, Theorem~\ref{thm:toolkit_robust} returns a constant~$C'$, for which the respective theorems hold. Let $\xi = \gamma/2$ and define the constant $C^* = \xi^{-1} \max \{C, C'\}$. Finally, let the graphs~$G$ and~$H$ be as given in the statement of Theorem~\ref{thm:tight_improvement}.

    Given the graph~$H$, let~$H_1$ be the subgraph of~$H$ formed by all the~$K_{\Delta+1}$ components in~$H$, and let $H_2 = H \setminus H_1$ be the remaining subgraph. Note that~$H_1$ and~$H_2$ are vertex disjoint subgraphs of~$H$. We will individually embed~$H_1$ and~$H_2$ into $G(p)$, for any given $p \geq C^*\, n^{-2/(\Delta+1)}(\log n)^{1/\binom{\Delta+1}{2}}$. Let $H_2'$ be the $n$-vertex graph obtained from $H_2$ by adding $|H_1|$ isolated vertices. Note that $m_1(H_2) = m_1(H_2')$, and hence as $H'_2$ is $K_{\Delta+1}$-free and $\Delta(H'_2) \leq \Delta$, Lemma~\ref{lem:suplementary} implies that $p \geq C^* n^{-1/m_1(H_2)} \log n$ for sufficiently large~$n$. 

    We first consider the case where $|V(H_1)| \leq \xi n$. Then we embed the graph~$H$ as follows. The graph~$G$ has a minimum degree~$\delta(G) \geq (\exth + \gamma)n$. By deleting at most $\Delta$ vertices, for $n > \Delta/\gamma$ one obtains an induced subgraph~$G'$ of~$G$ having minimum degree $\delta(G') \geq \exth n \geq \frac{\Delta}{\Delta+1}n$. Thus, Theorem~\ref{thm:toolkit_robust} returns a $K_{\Delta+1}$-factor in $G'(p)$. This gives a copy of~$H_1$ in~$G(p)$. Let this copy of~$H_1$ have a vertex set~$V_1$ in~$G(p)$. Now consider the induced subgraph $G_2 \coloneqq G[V(G) \setminus V_1]$ with $|G_2| = m$. Then, for any vertex $v \in V(G_2)$, we have that 
    $$d_{G_2}(v) \geq (\exth + \gamma)n - \xi n \geq (\exth + \gamma/2) n \geq (\exth + \gamma/2) m.$$ 
    Moreover, as $m \geq (1-\xi)n \geq \xi n $ and as $C^* \geq \xi^{-1}C$, we have
    \[p \geq C^* n^{-1/m_1(H_2)} \log n \geq C(\xi n)^{-1/m_1(H_2)}\log n \geq Cm^{-1/m_1(H_2)}\log m.\]
    
    Thus, by Theorem~\ref{thm:Robust_main}, the random subgraph $G(p)\bigl[V(G_2)\bigr] \sim G_2(p)$ contains a spanning copy of~$H_2$ on the vertex set~$V(G)\setminus V_1$. This, together with the copy of~$H_1$ in $G(p)[V_1]$, gives a copy of $H$ in $G(p)$, as required. Similarly, if $|V(H_2)| \leq \xi n$, then we can find a copy of $H$ in $G(p)$ by adding isolated vertices to $H_2$ and reversing the order of application of the theorems above. 

    Hence, we may assume that $|V(H_1)| > \xi n$ and $|V(H_2)| > \xi n$. Suppose that $|V(H_1)| = k$, and let $\mathcal{U} \subseteq V(G)$ be a set of $k$ vertices chosen uniformly at random without replacement. We shall say that a vertex~$v \in \mathcal{U}$ is \emph{bad}, if $\bigl|N(v; \mathcal{U})\bigr| < (\exth + \gamma/2)k$. For any vertex $v \in \mathcal{U}$, let~$X_v$ denote the size of its neighbourhood in~$\mathcal{U}$. Then,~$X_v$ is a hypergeometric random variable with parameters $(n, d_G(v), k)$ and expected value $\mathbb{E}\bigl(X_v) = d_G(v)k/n \geq (\exth + \gamma)k$. Thus, by an application of Lemma~\ref{lem:chern_hypgeom} with $t = \gamma k/2$ and $\eps = \gamma/2$, for any vertex $v \in \mathcal{U}$, we have that 
    $$\mathbb{P}(\text{$v$ is bad}) \leq \mathbb{P}\bigl(|X_v - \mathbb{E}(X_v)| \geq \gamma k/2 \bigr) \leq 2e^{-\gamma^3 k/24} \leq 2e^{-\gamma^3 \xi n/24}.$$
    Hence, with probability at least $1 - 2k\,e^{-\gamma^3 \xi n/24}$, the induced subgraph $G[\,\mathcal{U}\,]$ has minimum degree $\delta\bigl(G[\,\mathcal{U}\,]\bigr) \geq (\exth + \gamma/2)k$. Moreover, since $k > \xi n$ and since $C^* \geq \xi^{-1}C'$, it follows that $p  \geq C'k^{-2/(\Delta+1)}(\log k)^{1/\binom{\Delta+1}{2}}.$
    Thus by Theorem~\ref{thm:toolkit_robust}, with high probability the graph $G[\,\mathcal{U}\,](p)$ contains a $K_{\Delta+1}$-factor, and thereby a copy of~$H_1$ as well. 
    
    Similarly for the graph~$H_2$ and the complement set of vertices~$\mathcal{U}^c \subseteq V(H)$, repeating the above argument, \emph{mutatis mutandis}, it follows that with probability $1 - 2(n-k)e^{-\gamma^3 \xi n/24}$, the random subgraph $G[\,\mathcal{U}^c](p)$ contains a copy of~$H_2$ with high probability. Thus, with probability at least $1 - 2n\,e^{-\gamma^3 \xi n/24}$, Theorems \ref{thm:toolkit_robust} and \ref{thm:Robust_main} are applicable to the induced subgraphs~$G[\,\mathcal{U}\,]$ and~$G[\,\mathcal{U}^c]$ respectively, thereby returning a copy of~$H$ in~$G(p)$ with high probability, as was required.  
\end{proof}

\bibliographystyle{ams_edited}
\bibliography{Robust_Sauer_Spencer}
%\bibliography{sauerspencer-mrefed}
\end{document}